\documentclass[12pt,a4paper]{article}
\usepackage{amssymb}
\usepackage{amsmath,amsfonts,amsthm}
\usepackage{amstext,latexsym,bm}
\usepackage{graphicx,color}
\usepackage{amsmath}
\usepackage{amsfonts}
\usepackage{amsthm,amscd}
\usepackage{graphicx}
\usepackage{amsmath}
\usepackage{amssymb}
\usepackage{amstext}
\usepackage{amssymb,amsthm,scalefnt,mathrsfs,color,amsmath,multirow,float}

\newtheorem{theorem}{Theorem}
\newtheorem{corollary}[theorem]{Corollary}

\newtheorem{example}[theorem]{Example}
\newtheorem{lemma}[theorem]{Lemma}

\newtheorem{proposition}[theorem]{Proposition}
\newtheorem{remark}[theorem]{Remark}

\numberwithin{equation}{section}
\numberwithin{theorem}{section}

\begin{document}

\title{Thermodynamic Formalism for a family of cellular automata and duality with the shift}
\author{Artur O.  Lopes,  Elismar R. Oliveira and Marcelo Sobottka}
\maketitle

\begin{abstract}We will consider a family of cellular automata $\Phi: \{1,2,...,r\}^\mathbb{N}\circlearrowright$ that are not of algebraic type.
Our first goal is to determine conditions that result in the identification of probabilities that are at the same time $\sigma$-invariant and $\Phi$-invariant, where $\sigma$ is the full shift. Via the use of versions of the Ruelle operator $\mathcal{L}_{A,\sigma}$ and $\mathcal{L}_{B,\Phi}$ we will show that there is an abundant set of measures with this property; they will be equilibrium probabilities for different Lispchitz potentials $A,B$ and for the corresponding dynamics $\sigma$ and $\Phi$. Via the use of a version of the involution kernel $W$ for a $(\sigma,\Phi)$-mixed skew product $\hat{\Phi}: \{1,2,...,r\}^\mathbb{Z}\circlearrowright$, given $A$ one can determine $B$, in such way that the integral kernel $e^W$ produce a duality between eigenprobabilities $\rho_A$ for $(\mathcal{L}_{A,\sigma})^*$ and eigenfunctions $\psi_B$ for $\mathcal{L}_{B,\Phi}$. In another direction, considering the non-mixed extension $\hat{\Phi}_n : \{1,2,...,r\}^\mathbb{Z}\circlearrowright$ of $\Phi$, given a Lispchitz potential $\hat{A} : \{1,2,...,r\}^\mathbb{Z}\to \mathbb{R}$, we can identify a Lipschitz potential $A:\{1,2,...,r\}^\mathbb{N} \to \mathbb{R} $, in such away that relates the variational problem of $\hat{\Phi}_n$-Topological Pressure for $\hat{A}$ with the $\Phi$-Topological Pressure for $A$. We also present a version of Livsic's Theorem.
Whether or not $\Phi$ (or $\hat{\Phi})$ can eventually be conjugated with another shift of finite type is irrelevant in our context.
\end{abstract}

\section{Introduction} \label{sec:2}

Denote the alphabet by $\mathcal{A}=\{1,2,...,r\}$, $\Omega= \mathcal{A}^\mathbb{N}$, and take two Lispchitz potentials $B_1,B_2:\Omega \to \mathbb{R}.$

Let's first consider a more general setting that involves two commuting transformations $\Phi_1:\Omega \to \Omega$ and $\Phi_2:\Omega \to \Omega $, and where some of the results we will describe can be applied. A particular case of interest is when $\Phi_1$ is the unilateral shift $\sigma:\Omega \to \Omega $ and $\Phi_2:\Omega \to \Omega$ is a cellular automata to be defined later. 
As usual
$$\sigma(x_1,x_2,...,x_n,...) =(x_2,x_3,...,x_n,...).$$

The bilateral shift is denoted by $\hat{\sigma}: \mathcal{A}^\mathbb{Z}\to \mathcal{A}^\mathbb{Z}.$

Consider two $r$ to $1$ surjective expanding transformation $\Phi_1,\Phi_2$, which commute, and for $j=1,2$, the two corresponding topological pressure problems
\begin{equation} \label{tryp} P_{\Phi_j}(B_j):=\sup_{\rho \in\mathcal{M} (\Phi_j)} \{
h_{\Phi_j}(\rho)+ \,\int B_j (x) d \rho(x)\},
\end{equation}
where, $h_{\Phi_j}(\rho)$ is the Shannon-Kolmogorov entropy of $\rho$ for the dynamics of $\Phi_j$, and $\mathcal{M}(\Phi_j)$ and $\mathcal{M}(\Phi_1,\Phi_2)$ are the set of probabilities that are invariant, respectively, for $\Phi_j$ and both $\Phi_1, \Phi_2$. From the above problems we get, respectively, the solutions $\mu_{B_j,\Phi_j}$, $j=1,2$, which are called the equilibrium probabilities for respectively $B_j,\Phi_j$, $j=1,2$. The value $P_{\Phi_j}(B_j)$ is called the topological pressure for the potential $B_j$ and the dynamics of $\Phi_j$, $j=1,2$.

We will show here that $\mathcal{M}(\Phi_1,\Phi_2)$ is not empty, which will imply that
$ \mathcal{M}(\sigma,\Phi)$ is also not empty (and has an abundance of elements as described for instance in Example \ref{kli}).

The Ruelle operator is defined as $\mathcal{L}_j(f)=\mathcal{L}_{B_j, \Phi_j}(f)=g$, via
\begin{equation} \label{tryp33}g(y)=\mathcal{L}_j (f)(y) = \sum_{\Phi_j(x)=y} e^{B_j(x)} f(x).
\end{equation}

The Ruelle Theorem is valid for both $\mathcal{L}_1,\mathcal{L}_2$, and we get, respectively,
the eigenvalues $\lambda_{B_j, \Phi_j}$, eigenfunctions $\psi_{B_j, \Phi_j },$ and eigenprobabilities $\rho_{ B_j, \Phi_j}$ for $\mathcal{L}_j^*$, $j=1,2$. It is well known that
$\mu_{ B_j, \Phi_j}=\psi_{B_j, \Phi_j}\,\rho_{ B_j, \Phi_j} $,
when normalized.
A particular case of interest is when $\Phi_1=\sigma$ and $\Phi_2=\Phi$ is the CA to be defined in the introduction. 


A general reference for classical Thermodynamic Formalism is \cite{PP} and for IFS Thermodynamic Formalism see \cite{Fan}, \cite{Ye}, or \cite{BOR23}.

It is well known that the eigenfunction $\psi_{B_j, \Phi_j },$ and eigenprobability $\rho_{ B_j, \Phi_j}$ are dual objects; as the duality of functions and measures, but by this we mean a relation not in the sense of Riesz Theorem. They are obtained respectively via $\mathcal{L}_{B_j, \Phi_j}$ and $(\mathcal{L}_{B_j, \Phi_j})^*$. We are interested here, among other things, in conditions on $B_1$ and $B_2$ that relate $\psi_{B_2, \Phi_1 }$ and $\rho_{B_1, \Phi_2 }$.

In another direction, we ask: is there $B_1,B_2$ such that
$\mu_{ B_1, \Phi_1}= \mu_{ B_2, \Phi_2}?$

Following \cite{Kim}, the answer to this last question regards considering a certain special relation between $B_1$ and $B_2$.

Below, in order to be in accordance with the notation in \cite{Kim}, we will take $B_1=A_2$ and $B_2=A_1$.

The authors in \cite{Kim} introduced the condition
\begin{equation} \label{klj1} 1) A_1 - A_1 \circ \Phi_1 = A_2 - A_2 \circ \Phi_2,
\end{equation}
or a more general one, where it is assumed that there exists a Lipschitz function $w$ such
\begin{equation} \label{klj2} 2) (A_1 - A_1 \circ \Phi_1)- ( A_2 - A_2 \circ \Phi_2) = (w - w \circ \Phi_1) - (w - w \circ \Phi_1) \circ \Phi_2 ,
\end{equation}
which is equivalent to
$$ (A_1 - A_1 \circ \Phi_1) = [A_2 + ( w - w \circ \Phi_1 )] - [A_2 + ( w - w \circ \Phi_1 )] \circ \Phi_2 . $$

In \cite{Kim} it was shown that if $A_1,A_2$ satisfies condition 2), then
$\mu_{ A_2, \Phi_1}= \mu_{ A_1, \Phi_2}$ (see our Theorem \ref{maint} for a more detailed proof).

However, \cite{Kim} does not provide an example of potentials satisfying such conditions 1) or 2). An interesting question is to find an example where $\Phi_1=\sigma$ and $\Phi_2=\Phi$, and this will be provided later in Example \ref{kli}.

\smallskip

In Section \ref{sec:2} we introduce the local rule that will define the class of cellular automata $\Phi$ (and we present examples) that will be the focus of Sections \ref{ther} and \ref{dul}.

In Theorem \ref{maint} in Section \ref{ther} we exhibit conditions on potentials $A_1$ and $A_2$ such that
the respective equilibrium probabilities for $\sigma$ and $\Phi$, are the same.
Theorem \ref{maint1} is a
kind of reciprocal of Theorem \ref{maint}.

In Section \ref{dul} via the use of a version of the involution kernel $W$ for a $(\sigma,\Phi)$-mixed skew product $\hat{\Phi}: \{1,2,...,r\}^\mathbb{Z}\circlearrowright$, we show that given a Lipschitz potential $A$ one can determine another Lipschitz potential $B$, in such way that the integral kernel $e^W:\{1,2,...,r\}^\mathbb{Z} \to \mathbb{R}$ produce a duality between eigenprobabilities $\rho_A$ for $(\mathcal{L}_{A,\sigma})^*$ and eigenfunctions $\psi_B$ for $\mathcal{L}_{B,\Phi}$. In another direction, considering the non-mixed extension $\hat{\Phi}_n : \{1,2,...,r\}^\mathbb{Z}\circlearrowright$ of $\Phi$, given a Lispchitz potential $\hat{A} : \{1,2,...,r\}^\mathbb{Z}\to \mathbb{R}$, we can identify a Lipschitz potential $A:\{1,2,...,r\}^\mathbb{N} \to \mathbb{R} $, in such away that relates the variational problem of $\hat{\Phi}_n$-Topological Pressure for $\hat{A}$ with the $\Phi$-Topological Pressure for $A$. Properties for the $\hat{\Phi}_n$-equilibrium probability for $\hat{A} : \{1,2,...,r\}^\mathbb{Z}\to \mathbb{R}$ can be derived in this way.

In Section \ref{Livvi} we present a version of Livsic's Theorem for our setting.

In Section \ref{ss6}, with the aim of putting our results in context, we present a review of previous results that are at the interface of cellular automata and ergodic theory. 
In particular questions that are in one way or another related to Furstenberg's conjecture.

Finally, in the Appendix we present the proof of some technical results mentioned before on the text.

\section{A family of Cellular Automata}\label{sec:2}

Remember that the alphabet is $\mathcal{A}=\{1,2,...,r\}$ and $\Omega= \mathcal{A}^\mathbb{N}.$

In $\Omega$ we consider the metric $d$ such that for
$$x=(x_1,x_2,....), y=(y_1,y_2,....):$$

a) If $x_1\neq y_1$, then, $d(x,y)=2^{-1}=1/2,$
and otherwise,

\[d(x,y)=
\left\{
\begin{array}{ll}
0, & x=y \\
2^{-\{\min_j x_j \neq y_j\}}, & x\neq y
\end{array}
\right.
\]
$x=(x_1,x_2,....), y=(y_1,y_2,....).$

Note that, if $x_1=y_1$, then $d(x,y)\leq 1/4$.


We define a particular local rule $\phi: \mathcal{A} \times \mathcal{A} \to \mathcal{A}$, where we assume that
for any fixed $a$, the law 
\begin{equation} \label{perm} b\in \mathcal{A} \to \phi(a,b) \,\,\text{is bijective}. 
\end{equation} 
Then, for each $a$ we get that $b \to \phi(a,b)$
is a
permutation on $d$ symbols.

Since $\mathcal{A}$ is finite, the local rule is completely described by a matrix $M:=(\phi(i,j))_{r\times r}$ such that $\phi(a,b)=M_{a,b}$, thus each row is a permutation of $\mathcal{A}$. For example, for $\mathcal{A}=\{1,2,3\}$ we may choose,
\[M:=
\left(
\begin{array}{ccc}
1 & 2 & 3 \\
1 & 2 & 3 \\
3 & 1 & 2 \\
\end{array}
\right)
\]
So, in this case $\phi(3,2)=M_{3,2}=1$.


We will consider the cellular automata $\Phi : \Omega \to \Omega$ given by
\begin{equation} \label{tryp1}\Phi(x_1,x_2,x_3,x_4,..,x_n,...) = ( \phi(x_1,x_2), \phi(x_2,x_3), \phi(x_3,x_4),...).
\end{equation}

We avoid the case of the trivial cellular automata: for each $a$, $\phi(a,b)=b$, in order that $\Phi$ is not $\sigma$.
 Here we  call {\bf permutative}  the CA $\Phi$ obtained from the $\phi$ of \eqref{perm}. Our examples do not necessarily fit the ones in the family of algebraic cellular automata.

Note that $\sigma \circ \Phi = \Phi \circ \sigma.$ The  $\Phi$ defined by \eqref{perm} is in some sense the simpler of all possible CA.

We point out that our main goal is to determine properties that result in the identification of probabilities that are at the same time $\sigma$-invariant and $\Phi$-invariant.
If $A_2$ and $A_1$ are such that
\begin{equation} \label{klj26} (A_1 - A_1 \circ \sigma)- ( A_2 - A_2 \circ \Phi) = (w - w \circ \sigma) - (w - w \circ \sigma) \circ \Phi,
\end{equation}
for some Lispchitz function $w$, then, in Theorem \ref{maint} we will show that the classical $\sigma$-equilibrium probability for the pressure of $A_2$ (see \cite{PP}) will be also $\Phi$-invariant.

Therefore, whether or not $\Phi$ can eventually be conjugated with another shift is irrelevant in our context. We will work with information and properties obtained directly from $A_1, A_2, \sigma$ and $\Phi$ (see for instance expressions \eqref{two} and \eqref{three}).

The next result will show that $\Phi$ is such that for each $x\in \Omega$, the set of preimages $\Phi^{-n} (x)$, $n \in \mathbb{N}$ is dense in $\Omega$.



\begin{proposition} Let $\phi:\mathcal{A}\times\mathcal{A} \to \mathcal{A}$ a local rule (we assume in this section that
\begin{equation}\label{pkit}j \to \phi(a, j)\,\,\text{ is a bijection for any}\,\, a \in \mathcal{A}).\end{equation} Consider the map $\Phi: \Omega \to \Omega$, where $\Omega=\{1,\ldots,r\}^{\mathbb{N}}$, associated to this rule.\\
For any $x \in \Omega$ the set of preimages by $\Phi$ is dense.
\end{proposition}

For proof see Proposition \ref{des} in the Appendix.

We denoted by $\mathcal{M}(\sigma), \mathcal{M}(\Phi) $ and $\mathcal{M}(\sigma,\Phi)$ the set of probabilities that are invariant, respectively, for $\sigma, \Phi$ and simultaneously for $\sigma$, and $\Phi$.

The function $\Phi$ is continuous, expanding at rate $2$ (see Lemma \ref{thm:Phi expanding lemma}), and we will show that the set $\mathcal{M} (\Phi,\sigma)$ is not empty and with cardinality bigger than one in several cases (see Example \ref{kli}).

Given $a$, we denote $u_a(b) = c$, the element $c$ such that $ \phi(a,c)=b.$

Given $y=(b_1,b_2,...,b_n,..)$, {\bf there exist $r$ points} $x=(a_1,a_2,...,a_n,..)$, such that $\Phi(x)=y$. Indeed, given $a_1\in \mathcal{A}$, take
\begin{equation} \label{tryp2} x= (a_1,u_{a_1}(b_1),u_{ u_{a_1}(b_1)}(b_2),...) .
\end{equation}

Given $j\in \mathcal{A}$, we denote by $\tau_j:\Omega \to \Omega $ the function such that given $x=(x_1,x_2,...,x_n,..)$
$$\tau_j(x):=(j,u_{j}(x_1),u_{ u_{j}(x_1)}(x_2),...).$$

For each $j$ we get that $\tau_j$ has a Lipschitz constant equal to $1/2$. Indeed, note that if $w_1=(x_1,x_2,...,x_n, x_{n+1},x_{n+2},...)$ and $w_2=(x_1,x_2,...,x_n, y_{n+1}, y_{n+2}..)$, where $x_{n+1}\neq y_{n+1}$ (which means $d(w_1,w_2)= 2^{-n})$, then, as $\phi(x_{n},x_{n+1})\neq \phi(x_{n},y_{n+1})$ (by bijectivity on the second variable)

$$ d( \Phi (w_1), \Phi (w_2) =$$
$$ d( ( \phi(x_1,x_2), ..., \phi(x_{n},x_{n+1}),...) , ( \phi(x_1,x_2), ..., \phi(x_{n},y_{n+1}),...))=$$
$$ 2^{-(n-1)} = 2\, 2^{-n}= 2\, d(w_1,w_2).$$


\begin{remark} \label{rre} The map $\Phi: \Omega \to \Omega$ of \eqref{pkit}, where $\Omega=\{1,\ldots,r\}^{\mathbb{N}}$,
has a transitive orbit (see Theorem \ref{thm:transitivity}). In Theorem \ref{pepe} in Subsection \ref{Fi} we show that 
each periodic point $x$ for $\Phi$ with period $m$ is a solution of $\tau_{j_{1}}\circ\cdots\circ\tau_{j_{m}}(x)=x$, for some choice of $j_{1}, \ldots, j_{m} \in \{1, \ldots, r\}$.  For each $m$ there exists $r^m$ points of period $m$. In Subsection \ref{Fi} we exhibit the periodic points of periods two and three in a particular example.

\end{remark}

\begin{example} We will present examples of periodic orbits for $\Phi:\{1,2\}^\mathbb{N}\to \{1,2\}^\mathbb{N}$ in Subsection \ref{Fi}. For example,
when
$$ M =\left(
\begin{array}{cc}
\phi(1,1) & \phi(1,2) \\
\phi(2,1) & \phi(2,2)
\end{array}
\right)= \left(
\begin{array}{cc}
2 & 1 \\
1 & 2
\end{array}\right).$$

The points $x=(1,2,2,..,)$ and $(2,2,2,...)$ are fixed points.

The point $x=(2,1,2,2,2,2,..)$ has period two: $\Phi(x)=(1,1,2,2,2,..)$.


\end{example}

The next example shows that the sets $\mathcal{M}(\sigma)$ and $\mathcal{M}(\Phi)$ are different.

\begin{example} \label{eds} Consider $r=3$ and a case where lines can repeat as in
$$ M=\left(
\begin{array}{ccc}
\phi(1,1) & \phi(1,2) & \phi(1,3) \\
\phi(2,1) & \phi(2,2) & \phi(2,3) \\
\phi(3,1) & \phi(3,2) & \phi(3,3)
\end{array}
\right)= \left(
\begin{array}{ccc}
1 & 2 & 3 \\
1 & 2 & 3 \\
3 & 2 & 1
\end{array}
\right).$$

In this case
$$ \Phi^{-1} (\overline{1}) = \overline{3 3} \cup \overline{2 1} \cup \overline{1 1},$$
$$ \Phi^{-1} (\overline{2}) = \overline{3 2} \cup \overline{2 2} \cup \overline{1 2}$$
and
$$ \Phi^{-1} (\overline{3}) = \overline{3 1} \cup \overline{2 3} \cup \overline{1 3}.$$

Moreover, $\Phi^{-2} (\overline{1})\cup \Phi^{-2} (\overline{2}) \cup \Phi^{-2} (\overline{3})$ is a total of $3^3$ different cylinders.

The independent probability $\mu$ on $\{1,2,3\}^\mathbb{N}$ associated to the weights
$$(1/7,2/7,4/7)$$ is $\sigma$-invariant but not $\Phi$-invariant. Indeed,
$$\mu (\Phi^{-1} (\overline{1}))= \mu(\overline{3 3}) + \mu(\overline{2 1}) + \mu(\overline{1 1})=$$
$$ 4/7 \,\,\, 4/7 + 2/7 \,\,\, 1/7 + 1/7 \,\,\, 1/7 \neq 4/7 \,\,\, 1/7 + 2/7 \,\,\, 1/7 + 1/7 \, \,\, 1/7 =1/7= \mu (\overline{1}).$$

In this case $\mathcal{M}(\sigma) \neq \mathcal{M}(\Phi).$

If $\mu_0$ is the measure of maximal entropy (for the shift $\sigma$), then,
\begin{equation} \label{oi1} \mu_0 (\overline{b_1,b_2,...,b_n} )= 3^{-n}
,\end{equation}
and therefore from the above $\mu_0$ is also $\Phi$-invariant, that is $\mu_0 \in \mathcal{M}(\sigma,\Phi) $.

\medskip

 A particular case of $\phi$ as in \eqref{perm} is the {\bf bipermutative} CA $\Phi$ where we assume the extra assumption on $\phi$:
 for  each  $b$, the law 
\begin{equation} \label{biperm} a\in \mathcal{A} \to \phi(a,b) \,\,\text{is bijective}. 
\end{equation} 

Most of our results are for the general case of the permutative CA $\Phi$.
We point out that in several papers on  the area  what is called  a bipermutative CA is always described by a local rule $\phi$ derived from  some kind of  group structure on the set of symbols $\{1,2,...,r\}$; the algebraic cellular automata (see for instance \cite{Pivato2005} and \cite{Haw}).

\end{example}

\begin{example} \label{kli} We will provide an example where equation \eqref{klj1} is true for $\Phi_1=\sigma$ and $\Phi_2=\Phi$, $r=2$, for functions that depends on the first two coordinates.

Consider $\phi$ such that
$$M= \left(
\begin{array}{cc}
\phi(1,1) & \phi(1,2) \\
\phi(2,1) & \phi(2,2)
\end{array}
\right)= \left(
\begin{array}{cc}
2 & 1 \\
2 & 1
\end{array}
\right).$$

This example of $\phi$ defines   a permutative $\Phi$ that is not bipermutative.

Take the functions $A_1$ and $A_2$ depending on the two first coordinates satisfying
$$ A_1 (x_1,x_2,...) = Q_{x_1,x_2} $$
and
$$ A_2 (x_1,x_2,...) = C_{x_1,x_2} .$$

Assume that
$$ Q_{2,1} = Q_{1,2}, C_{1,2} = C_{1,1} + Q_{1,2} - Q_{2,2}, $$
\begin{equation} \label{two}C_{2,1} =C_{1,1} - Q_{1,1} + Q_{1,2}, C_{2,2} = C_{1,1}.
\end{equation}

Then,

$$ A_1 - A_1 \circ \sigma = A_2 - A_2 \circ \Phi.$$

From \cite{Kim} (see also Theorem \ref{maint}) this implies that the $\sigma$-equilibrium for $A_2$ is equal to the $\Phi$-equilibrium for $A_1$. Then, $\mathcal{M}(\sigma,\Phi)$ is not empty.

Without loss of generality we can assume that $C_{12}= 1 - C_{11}$ and $C_{22}= 1 - C_{21}$ (see \cite{Lo1}).
The $\sigma$-equilibrium probability for $A_2$ is a Markov stationary measure $\mu$ (see \cite{Lo1} or \cite{Sp}). The set $\mathcal{M} (\sigma, \Phi)$ contains the one-parameter family of stationary Markov probabilities $\mu$ index by $-1<\lambda<1$, obtained by symmetric line stochastic matrices, where $C_{2,1} - C_{1,1,} = \lambda$. In this case $Q_{1,1} = Q_{1,2} - \lambda$ and $Q_{2,2}=Q_{1,1}, Q_{1,2}=Q_{2,1}.$

For the line stochastic matrix
$$ \left(
\begin{array}{cc}
1/3 & 2/3 \\
2/3 & 1/3
\end{array}
\right),$$
the left invariant probability vector is $(\Pi_1,\Pi_2) = (1/2,1/2).$

In this case
$$\Phi^{-1} (\overline{1}) = \overline{1,2} \cup \overline{2,2},$$
and
$$\Phi^{-1} (\overline{2}) = \overline{1,1} \cup \overline{2,1}.$$

Then,
$$\mu (\Phi^{-1} (\overline{1})) = \mu (\overline{1,2}) + \mu (\overline{2,2}) = 1/2 \,\, 2/3 + 1/2 \,\,1/3 = 1/2 = \mu ( \overline{1}).$$

In the same way $\mu (\Phi^{-1} (\overline{2})) = \mu ( \overline{2}) .$ This, computation highlights the property that $\Phi$ preserves the Markov probability $\mu$.

\smallskip

If
$$ M =\left(
\begin{array}{cc}
\phi(1,1) & \phi(1,2) \\
\phi(2,1) & \phi(2,2)
\end{array}
\right)= \left(
\begin{array}{cc}
2 & 1 \\
1 & 2
\end{array}
\right),$$
there are also solutions for $A_1$ and $A_2$ when
$$Q_{2,1} = 2 Q_{1,2} - Q_{1,1}, Q_{2,2} = Q_{1,2}, C_{1,2} = C_{1,1} + Q_{1,1} - Q_{1,2},$$
$$C_{2,1} = C_{1,1} - Q_{1,1} + Q_{1,2}, C_{2,2} = C_{1,1}. $$

In this case, one can deduce (after a simple manipulation) that the independent probability with weights $(1/2,1/2)$ (the maximal entropy measure) is the only one that can be obtained as $\Phi$-equilibrium for potentials that depend on two coordinates. 

This example of $\phi$ defines   a bipermutative  CA $\Phi$, and it also fits the case of an  algebraic CA on two symbols (see \cite{Pivato1} and \cite{Pivato2005}).

The property that we just show for this bipermutative CA (the maximal entropy measure is the only one in $\mathcal{M} (\sigma,\Phi),$ for potentials that depend on two coordinates)  is in some sense a  very particular example of a series of important cases covered by   the setting of  algebraic cellular automata (see \cite{Sob2}, \cite{Sobottka2008}, \cite{CS} and \cite{Pivato1}).

For instance in \cite{HostMaassMartinez} the authors show that 
the only $\Phi$-invariant, $\sigma$-ergodic measure $m$ with with positive  $\Phi$-entropy is the maximal entropy measure.\\
In another direction for affine cellular automata  the authors in \cite{mmpy} show that the only measure with complete connections and summable decay, that is simultaneously invariant by the cellular automata and the shift map, is the measure of maximum entropy.

\end{example}

We will provide an example where equation \eqref{klj1} is true for $\Phi_1=\sigma$ and
$\Phi_2=\Phi$, $r=2$, for functions that depend on the first three coordinates in Example \ref{kli45}
in the Appendix. This will show the existence of more complex examples in $\mathcal{M}(\sigma,\Phi)$.


\begin{proposition} \label{ell}
Consider $\mathcal{K}:=\{\overline{1}, \overline{2}, \ldots, \overline{r}\}$, a partition of $\Omega$. Then, for all $n \geq 0$ we have
$$ \mathcal{K} \bigvee \Phi^{-1}(\mathcal{K}) \bigvee \ldots \bigvee \Phi^{-n}(\mathcal{K})=\biguplus_{a_{1} \ldots a_{n+1} \in \mathcal{A}} \overline{a_{1} \ldots a_{n+1}}.$$

Moreover, give $n$, $x\in \Omega$, for each $y=(a_1,a_2,..,a_n,y_{n+1},y_{n+2},..) \in \Phi^{-n}(x)$, there exist a unique $z$ of the form $z=(a_1,a_2,...,a_n, z_{n+1},z_{n+2},..)$ in $\sigma^{-1} (x)$, and vice versa.
\end{proposition}

For proof see Proposition \ref{ell43} in the Appendix.

\smallskip

\begin{remark} \label{ccon} Given the natural partition $\mathcal{K}=\{\overline{1}, \overline{2},
..., \overline{r}\}$
$$\mathcal{K}_n= \mathcal{K} \vee \Phi^{-1} \mathcal{K} \vee ...\vee \Phi^{-n}\mathcal{K} $$
suppose that it is equal to
the partition $\{\overline{a_1,a_2,...,a_n}, a_j\in \mathcal{A}, j=1,2,...,n\}$, for every $n \in \mathbb{N}$. This is the case we consider here as shown in Proposition \ref{ell}.
\end{remark}

In the case $\mu \in \mathcal{M}(\sigma)$, then the entropy according to $\sigma$
$$ h(\mu)=h_\sigma (\mu)=$$
\begin{equation} \label{utum} -\,\lim_{m \to \infty} \frac{1}{m} \sum_{ (a_1,a_2,..,a_m) \in \{1,2,...,d\} ^m} \mu(\overline{a_1,a_2,..,a_m}) \log (\mu(\overline{a_1,a_2,..,a_m})).
\end{equation}

In the case $\mu \in \mathcal{M}(\Phi)$, then the entropy according to $\Phi$ is
\begin{equation} \label{utum} h(\mu)=h_\Phi(\mu)= -\,\lim_{n \to \infty} \frac{1}{n} \sum_{ S \in \mathcal{K}_n} \mu(S) \log (\mu(S)).
\end{equation}

In this case, if $\mu \in \mathcal{M}(\sigma,\Phi)$, then $ h_\Phi(\mu)= h_\sigma(\mu).$



\section{Thermodynamic Formalism for $\Phi$} \label{ther}

First will state some results on  Thermodynamic Formalism for the $\Phi$ defined via \eqref{perm}. We point out that the proofs  of these results  follow from  a simple adaptation of the ones claiming  analogous results (which are well known) for the shift $\sigma$, as in \cite{PP} ($\Phi$ is an expanding map as shown in Lemma  \ref{thm:Phi expanding lemma}). We point out that the IFS  Thermodynamic Formalism (as in \cite{Hu}, \cite{BaDe} and \cite{BaDeEl}) can also be applied to $\Phi$  (the similar results that we will use here are also true)

The classical Ruelle Theorem is valid for $\sigma$; the same proof also works replacing $\sigma$ for $\Phi$. In order to make a distinction we can write $\mathcal{L}_{A_2,\sigma}$ and  $\mathcal{L}_{A_1,\Phi}$.

Given a Lipschitz potential $A:\Omega = \{1,2,...,r\}^\mathbb{N} \to \mathbb{R}$,
we interested in a probability $\mu_A$ which maximizes the pressure
\begin{equation} \label{tryp} P_\Phi(A):=\sup_{\rho \in\mathcal{M} (\Phi)} \{h_\Phi(\rho)+  \int A(x) d \rho(x)\}=
\end{equation}
$$  h_\Phi(\mu_A) + \int A \, d \mu_A,$$
where $A$ is the potential and $h_\Phi(\rho)$ is the $\Phi$-entropy  of $\rho$ (defined before).

Such $\mu_A$ is unique and will be called   the $\Phi$-equilibrium for $A$.



Given a Lipschitz potential $A:\Omega \to \mathbb{R}$, define $\mathcal{L}_A(f)=\mathcal{L}_{A, \Phi}(f)=g$, via
\begin{equation} \label{tryp3}g(y)=\mathcal{L}_A (f)(y) = \mathcal{L}_{A,\Phi} (f)(y)=\sum_{\Phi(x)=y} e^{A(x)} f(x).
\end{equation}

Note that for each $n>1$
\begin{equation} \label{tryp5}g(y)=\mathcal{L}_A^n (f)(y) = \sum_{\Phi^n(x)=y} e^{\sum_{j=0}^{n-1} A(\Phi^j (x))} f(x).
\end{equation}

Consider a  Lipschitz-continuous potential $A$ with Lipschitz constant $c$, and $x,y\in \mathcal{A}^\mathbb{N}$,  then given ${\bf a}\in \mathcal{A}^k$, there exist just two points ${\bf a} x, {\bf a} y $ in the same cylinder set such that
$\Phi^n ({\bf a} x)=x,$ and $\Phi^n( {\bf a} y)=y $. An  important property for the validity of the Ruelle Theorem (that we can also use for the dynamics of $\Phi$) is
\begin{equation} \label{esque}
|A({\bf a} x) - A({\bf a} y)| \leq c \, d({\bf a} x, {\bf a} y) \leq  c\,\,\frac{1}{2^{k }}\,\,  d(x,y)
\end{equation}
Above we are using the second claim in Proposition \ref{ell}.


Note that
given a
Lipschitz-continuous potential $A$, there exists $b>0$, such that, $\forall \,n$, and $\forall \,{\bf j}=(j_1,j_2,...,j_n)\in \{1,2,...,r\}^n=  \mathcal{A}^n$

\begin{equation} \label{kkh} \frac{\prod_{j=0}^{n-1} e^{A (\Phi^j(x))} }{ \prod_{j=0}^{n-1} e^{A (\Phi^j(y))}}<b,
\end{equation}
$\forall \,x,y \in \overline{j_1,j_2,...,j_n} =\overline{{\bf j}}$.


From \eqref{esque} follows that for a Lipschitz-continuous potential $A$, given $x,y\in \mathcal{A}^\mathbb{N}$,  we get that
\begin{equation} \label{kkhau}| S_{n,  A}\,({\mathbf a}x) - S_{n,  A}\,({\mathbf a}x)|
\end{equation}
is bounded independently of $n$ and the size of the word ${\mathbf a}\in \mathcal{A}^k$, $k>0$.

The above property \eqref{kkh} (or \eqref{kkhau}) is called the {\bf bounded distortion property} for the Lipschitz-continuous  potential $A$. This property is the key ingredient for analyzing the asymptotic of \eqref{tryp5} when $n \to \infty$.

We denote by $\mathcal{C}=C(\Omega, \mathbb{R})$ the space of continuous functions $f:\Omega \to \mathbb{R}.$

For the proof of the next theorem see for instance  \cite{PP} (adapted to our setting).

\begin{theorem}\label{rer1}
There exists a strict positive Lipschitz eigenfunction
 $\psi_A=\psi_{A, \Phi}$
for $\mathcal{L}_{A,\Phi} : {\cal C} \to {\cal C}$, associated to a unique strictly positive eigenvalue
$\lambda_A=\lambda_{A,\Phi}$. The eigenvalue is simple (and isolated from the rest of the spectrum when $ \mathcal{L}_A $ acts on the set of Lipschitz functions) and it is equal to the supremum of the modulus of the values of the spectrum. Moreover, $\log \lambda_{A,\Phi} =P_\Phi (A).$
\end{theorem}

The pressure $P_\Phi (- \log r)=0$ because in this case, the eigenvalue is equal to $1$.

\begin{remark} \label{infa} If a continuous function $f>0$ satisfies for some $\lambda>0$
$$\mathcal{L}_A (f) = \lambda f,$$
then, $\lambda$ is the main eigenvalue and $f$ is the main eigenfunction (see \cite{PP}).

\end{remark}
Given a continuous potential $A:\Omega \to \mathbb{R}$, we  can define the dual operator $\mathcal{L}_A^*= \mathcal{L}_{A,\Phi}^*$ on the space of the Borel finite measures on $\Omega$, as the operator that sends a measure $\nu$ to the measure $\mathcal{L}_A^*(\nu)$,  defined by
\begin{equation} \label{feij2}\int \psi\, d \mathcal{L}_A^*(v) =  \int \mathcal{L}_A (\psi)\, dv\,.
\end{equation}
for any $\psi \in \mathcal{C}$.

The operator  $\mathcal{L}_A^*$ acts on the space of all probabilities in $\Omega$. Note that if $v $ is $\Phi$-invariant, then, not necessarily  the probability  $\mathcal{L}_A^*(v)$ is $\Phi$-invariant (the same claim is valid for $\sigma$).

\begin{lemma} Suppose $A:\Omega \to \mathbb{R}$ is a  Lipschitz potential. Then, there exist a probability $\rho_{A}=\rho_{A,\Phi}$ on $\Omega$ and a real positive eigenvalue $\tilde{\lambda}_{A}=\tilde{\lambda}_{A,\Phi}$ such that

\begin{equation} \label{maru} \mathcal{L}_A^*  ( \rho_A) =\tilde{\lambda}_A \, \, \rho_A.
\end{equation}

Moreover, $\tilde{\lambda}_A= \lambda_A$.
\end{lemma}

We call such $\rho_A$ the eigenprobability for $A$.

\smallskip

\begin{remark} \label{nno}The eigenprobability $\rho_A$ and the eigenfunction $\psi_A$ are dual entities, one is obtained from to $\mathcal{L}_A^*$ and the other is obtained from  $\mathcal{L}_A$; both for the same eigenvalue $\lambda_A$.
\end{remark}

\smallskip

One can show (see \cite{PP}) that $\psi_{A,\Phi} \, \rho_{A,\Phi}$ is the equilibrium probability for $P_\Phi(A).$ Moreover, if $w$ is such that $B= A + w - w \circ \Phi$, then the equilibrium probability for $B$ and $A$  are equal  and $P_\Phi(A)= P_\phi(B)$ (see \cite{PP}).

In \cite{Kim}  the authors considered transformations $\Phi_1,\Phi_2: \Omega \to \Omega$ which commute and we will adapt their results for our setting: $\Phi_1=\sigma$ and $\Phi_2=\Phi$.  For generality, we will state some claims in terms of $\Phi_1$ and $\Phi_2$ (assuming they satisfy the proper conditions).

\begin{remark} \label{ytre}Denote by $\Pi$ either $\Phi_1$ or $\Phi_2$.
We say that the Lipschitz potential $B$ is normalized for $\Pi$ if $\mathcal{L}_{B, \Pi}(1)=1.$

In this case
$$\lim_{n \to \infty } \mathcal{L}_{B, \Pi}^n ( f) = \int f d \mu,$$
where $\mu$ is the equilibrium potential for $B$ (see \cite{PP}). Moreover, $\mathcal{L}_{B, \Pi}^* (\mu)=\mu.$

If $B$ is normalized it is usual to denote by $J:\Omega \to (0,1)$ the function such that $B=\log J$. We call $J$ the Jacobian of the  probability $\mu$ such that $\mathcal{L}_{B, \Pi}^* (\mu)=\mathcal{L}_{\log J, \Pi}^* (\mu)=\mu.$

\end{remark}

In the case of $\Phi$, if $J$ is a Jacobian, then
for any $x$ we have
$$ \sum_{\Phi (y)=x} J(y)=1.$$
The Jacobian in some sense plays the role of a stochastic matrix. In Example 8 in \cite{Lo1} this is properly described for the case of the shift.

\smallskip

Given a Lipschitz potential $A$, the associated eigenvalue $\lambda_A$ and the corresponding  eigenfunction $\psi_A$, the potential
$$\bar A = A + \log \psi_A - \log (\psi_A \circ \Phi) - \log \lambda_A$$
is normalized, has pressure zero and it is called the normalization of $A$. The equilibrium probability for $A$ is the measure of maximal entropy, if and only if  $\bar A = -\log r$ (see the end of Section 2 in \cite{Lo1}); this is true for $\Phi$ and $\sigma$.

\smallskip

An interesting result claims the following (see Lemma 2.1 in \cite{Kim}):

\begin{lemma} \label{cafi} Suppose $\Phi_1$ and $\Phi_2$ commute, then
$$ \mathcal{L}_{A_1,\Phi_2} \circ  \mathcal{L}_{A_2,\Phi_1}=\mathcal{L}_{A_2,\Phi_1}  \circ \mathcal{L}_{A_1,\Phi_2},$$
if and only if,
$$A_1 - A_1 \circ \Phi_1 =  A_2 - A_2 \circ \Phi_2.$$

\end{lemma}

Below, in Theorem \ref{maint}, we will fill in some details missing on  the proof of Theorem 2.2  in  \cite{Kim}  that we believe would be appropriate.

The next theorem contemplates the case
$$ A_1 - A_1 \circ \sigma =  A_2 - A_2 \circ \Phi,$$
taking  $w=0.$ Therefore, $\mathcal{M} (\Phi,\sigma)$ is not empty for the case of Example  \ref{kli}.

 \smallskip

Given  Lipschitz potentials $A_1$ and $A_2$, we ask: what are sufficient conditions on $A_1,A_2$ in such way that  $\mu_{  A_2, \Phi_1 }= \mu_{  A_1, \Phi_2}$?

\begin{theorem} \label{maint}
 Suppose $\Phi_1$ and $\Phi_2$ commute, and assume that  $A_1$ and $A_2$ are Lipschitz functions.
 If for some Lipschitz function $w$ we get
\begin{equation}\label{tyr} (A_1 - A_1 \circ \Phi_1)- (  A_2 - A_2 \circ \Phi_2) =   (w - w \circ \Phi_1) -   (w - w \circ \Phi_1) \circ  \Phi_2 ,
\end{equation}
 Then,  $\mu_{  A_2, \Phi_1}= \mu_{  A_1, \Phi_2}$.

\end{theorem}

\begin{proof} Note that if $(B_1 - B_1 \circ \Phi_1) =  (B_2 - B_2 \circ \Phi_2) $, then for any $f$
\begin{equation} \label{tryp37}\mathcal{L}_{B_1, \Phi_2} ( \mathcal{L}_{B_2, \Phi_1}( f))  = \mathcal{L}_{B_2 + B_1 \circ  \Phi_1,\Phi_2 \circ \Phi_1} (f)(y)=\mathcal{L}_{B_2, \Phi_1} ( \mathcal{L}_{B_1, \Phi_2}( f))  .
\end{equation}

 Therefore, the Ruelle operators commute.

 Condition \eqref{tyr} implies that
 $ \mathcal{L}_{A_1,\Phi_2}$ and $ \mathcal{L}_{A_2 + ( w - w \circ \Phi_1),\Phi_1}$ commute.

 Suppose that $\varphi>0$ is eigenfunction for $\mathcal{L}_{A_1,\Phi_2}$, that is $\mathcal{L}_{A_1,\Phi_2}(\varphi) = \lambda \varphi$, and $\psi>0$ is eigenfunction for $\mathcal{L}_{A_2,\Phi_1}$, that is  $\mathcal{L}_{A_2,\Phi_1}(\psi) = \beta \psi.$

 Moreover, we set  $\mathcal{L}_{A_2,\Phi_1}^*(\rho) = \beta \rho$ and  $\mathcal{L}_{A_1,\Phi_2}^* (\nu) = \lambda \nu.$ The equilibrium probability for $A_2$ is $\psi\,\rho$.

 Without loss of generality, we can assume that $\beta=\lambda=1$.  Indeed. Replacing $A_1$ by $A_1 - \log \lambda$ and $A_2 $ by $A_2 - \log \beta$ we keep the relation \eqref{tyr} for the same $w$ (and, respectively, the equilibrium properties will not change). We can also assume that  $A_2$ is $\Phi_1$ normalized, that is $\psi=1,$
which means $\rho=  \mu_{  A_2, \Phi_1,1 }.$

 The eigenfunction for  $\mathcal{L}_{A_2 + ( w - w \circ \Phi_1),\Phi_1}$ is
 $f= \,e^{-w}$, indeed
 $$ \mathcal{L}_{A_2 + ( w - w \circ \Phi_1),\Phi_1} ( f)(x)=$$
$$ \mathcal{L}_{A_2 + ( w - w \circ \Phi_1),\Phi_1} ( \,e^{-w} )(x)=\sum_{ \Phi_1 (y)=x}
e^{A_2 (y)} \frac{ e^{ w(y)}}{e^{w \circ \Phi_1(y) }} e^{-w(y)} = $$
$$ \sum_{ \Phi_1 (y)=x}
e^{A_2 (y)} \frac{ e^{ w(y)}}{e^{w (x) }} e^{-w(y)}=  \sum_{ \Phi_1 (y)=x}
e^{A_2 (y)} \frac{ 1}{e^{w (x) }} =  \frac{ 1}{e^{w (x) }}\, . $$

 Now, note that
 $$ \mathcal{L}_{A_2 + ( w - w \circ \Phi_1),\Phi_1}  ( \mathcal{L}_{A_1,\Phi_2}(\frac{ 1}{e^{w }}\, )   )  = \mathcal{L}_{A_1,\Phi_2} ( \mathcal{L}_{A_2 + ( w - w \circ \Phi_1),\Phi_1}  (\frac{ 1}{e^{w }}\, ) \,) =$$
 $$  \mathcal{L}_{A_1 ,\Phi_2} ( \frac{ 1}{e^{w }}\, ). $$

 Therefore, $g= \mathcal{L}_{A_1 ,\Phi_2} ( \frac{ 1}{e^{w }}\, )  >0$ is eigenfunction for $\mathcal{L}_{A_2 + ( w - w \circ \Phi_1),\Phi_1} $, and associated to the eigenvalue $1$. From Remark \ref{infa} we get that $g$ is colinear with $\frac{ 1}{e^{w }}\, $. That is   $\mathcal{L}_{A_1 ,\Phi_2} ( \frac{ 1}{e^{w }}\, )  = \gamma \,
\frac{ 1}{e^{w }}\, $, for some $\gamma>0$.  Once more from  Remark \ref{infa} we get that $\gamma=1$, and the eigenfunction $\varphi$ for $\mathcal{L}_{A_1 ,\Phi_2}$ is colinear with $ \frac{ 1}{e^{w }}\, .$ This shows that $w=  - \log \varphi +c.$

This shows that \eqref{tyr} is true replacing $w$ by $- \log \varphi.$

Then,  $ \mathcal{L}_{A_1,\Phi_2}$ and $ \mathcal{L}_{A_2 + (\,( - \log \varphi)  -  (- \log \varphi)  \circ \Phi_1),\Phi_1}$ commute.

Expression \eqref{tyr} is equivalent to

$$  (A_2 - A_2 \circ \Phi_2)- (  A_1 - A_1 \circ \Phi_1) =$$
\begin{equation}\label{tyro}    ((-w) - (-w) \circ \Phi_2) -   ((-w) - (-w) \circ \Phi_2) \circ  \Phi_1.
\end{equation}

Therefore,
 $ \mathcal{L}_{A_2,\Phi_1}$ and $ \mathcal{L}_{A_1   - w + w \circ \Phi_2,\Phi_2}$ commute.

Note that the potential $A_1 - w + w \circ \Phi_2 - \log \gamma= A_1 - w + w \circ \Phi_2 $ is $\Phi_2$ normalized.

From Remark \ref{ytre} we get for any continuous function $h$
$$\lim_{n \to \infty } \mathcal{L}_{A_2, \Phi_1}^n ( \mathcal{L}_{A_1   - w + w \circ \Phi_2- \log \gamma,\Phi_2})(h) = \int  \mathcal{L}_{A_1   - w + w \circ \Phi_2- \log \gamma,\Phi_2}(h)  d \rho,$$

From the commutative property, we get from Remark \ref{ytre}
$$\lim_{n \to \infty }   \mathcal{L}_{A_1   - w + w \circ \Phi_2- \log \gamma,\Phi_2} (\mathcal{L}_{A_2, \Phi_1}^n (h)) =\mathcal{L}_{A_1   - w + w \circ \Phi_2- \log \gamma,\Phi_2} (\int h d \rho)= \int h d \rho. $$

This shows that   $\mathcal{L}_{A_1   - w + w \circ \Phi_2,\Phi_2}^* (\rho)=\rho.$ Therefore, we were able to show that $\rho$ is the $\Phi_2$-equilibrium probability for $A_1$. That is,  $\mu_{  A_2, \Phi_1,1 }= \mu_{  A_1, \Phi_2,1 }.$

\end{proof}

\begin{remark} \label{difr}
$(A - A \circ \sigma)$ is in some sense the discrete time version of the $\sigma$-derivative of the ``function'' $A$. Note that in the case $\mu$ is $\sigma$-invariant for the shift we get $\int (A - A \circ \sigma) d \mu =0$, which corresponds to $\int_a^b f'(x) dx =0$, when $f$ is periodic in $[a,b]$.

The hypothesis
$$(A_1 - A_1 \circ \sigma)- (  A_2 - A_2 \circ \sigma) =   (w - w \circ \sigma) -   (w - w \circ \sigma) \circ  \sigma ,$$
is a discrete version of the postulation
$$f' - g' = \frac{d^2}{d x^2} w.$$

The hypothesis
$$(A_1 - A_1 \circ \Phi_1)- (  A_2 - A_2 \circ \Phi_2) =   (w - w \circ \Phi_1) -   (w - w \circ \Phi_1) \circ  \Phi_2 , $$
is in some sense a kind of mixed derivatives expression.

\end{remark}

The next result is a kind of converse of Theorem \ref{maint}.

\begin{theorem} \label{maint1}
 Suppose $\Phi_1$ and $\Phi_2$ commute, and assume that  $A_1$ and $A_2$ are Lipschitz functions.
 Also assume that $ f d \mu = g d \nu$, where $f,\mu$ are, respectively, the eigenfunction and eigen-probability for $\mathcal{L}_{A_1 ,\Phi_2} $, and $g,\nu$ are, respectively, the eigenfunction and eigenprobability for $\mathcal{L}_{A_2 ,\Phi_1} $.

 Then, the Lipschitz function $w= \log g - \log f$ satisfies
\begin{equation}\label{tyr689} (A_1 - A_1 \circ \Phi_1)- (  A_2 - A_2 \circ \Phi_2) =   (w - w \circ \Phi_1) -   (w - w \circ \Phi_1) \circ  \Phi_2 .
\end{equation}

\end{theorem}

For proof see  Theorem 2.2 in \cite{Kim}.

\begin{remark} \label{maint5} When asking properties derived from the equality between  two equilibrium probabilities  for
respectively, $A_1,\Phi_2$ and $A_2, \Phi_1$, we can assume in the hypothesis of Theorem  \ref{maint1} that $f=1=g$ (that is, $A_1$ is $\Phi_2$ normalized and $A_1$ is $\Phi_2$ normalized). In this way, from the hypotheses $\mu=\nu$, we derive the property
\begin{equation}\label{tyr659} (A_1 - A_1 \circ \Phi_1)= (  A_2 - A_2 \circ \Phi_2) .
\end{equation}

\end{remark}

\section{The involution Kernel:  $\sigma$ and $\Phi$ duality} \label{dul}

The eigenfunction and the eigenprobability are dual concepts (see Remark \ref{nno}) and in this section, we will address results in some way related to this claim.

 We  will adapt the reasoning of \cite{PP}, Proposition 1.2 in Section 1, to show that a natural skew dynamics associated with cellular automata (and the shift) have the special property that a potential $f(x,y)$,  initially defined at the skew structure $\{1,2, ..., k\}^{\mathbb{N}} \times \{1,2, ..., k\}^{\mathbb{N}}$, is co-homologous to another one $g(x,y)$ that depends only of the ``future" coordinates $x$. This will generalize the well known result for the  shift   $\hat{\sigma}$ acting on $\{1,2, ..., k\}^{\mathbb{Z}}$, as in \cite{PP}. Here we represent
$$\{1,2, ..., k\}^{\mathbb{Z}}\simeq \{1,2, ..., k\}^{\mathbb{N}} \times \{1,2, ..., k\}^{\mathbb{N}}:=\{(x,y) | x, y \in \{1,2, ..., k\}^{\mathbb{N}}\}$$
to avoid the identification of the coordinate correspondent to zero when we define our dynamics. In this notation, $x$ is the ``future" and $y$ is the ``past" because we could relabel $y$ as the negative coordinates in $\{1,2, ..., k\}^{\mathbb{Z}}$.

The  inverse branch $\tau_j:  \Omega \to \Omega$ is given by the formula
   $$\tau_j(x)=y \text{ if } \Phi(y)=x \text{ and  } y_1=j.$$

In the next result, we assume that $x \to A(x)$ is a function of the future variables $x$ and $y \to A^*(y)$ is a function of the past variable $y$. One of the purposes of Subsection \ref{uuy} is to derive such result from a more general reasoning which is of interest in itself.

\begin{proposition} \label{kjr35} Assume that $A$ depending just of the future coordinates $x$, and $A^*$ which depends on past coordinates $y$, are such that, there exist $W:\Omega \times \Omega$ satisfying for $i=1,2,...,r$, $x,y\in \Omega$
\begin{equation} \label{pit11}(A+W)(ix,y)= (A^*+W)(x, \tau_{i}(y)).
\end{equation}
Then, if $\rho_{A^*}$ is the eigenprobability for the $\Phi$-Ruelle operator $\mathcal{L}_{A^*,\Phi}$ of the potential $A^*$, then
$$\phi(x) = \int e^{W(x,y)} d \rho_{A^*}(y)$$
is the eigenfunction for the $\sigma$-Ruelle operator $ \mathcal{L}_{A,\sigma}$ of the potential $A$.

Moreover, given a Lipschitz function $y \to A^*(y)$, there exists a Lipschitz function $x \to A(x)$ and $W$ as in \eqref{pit11}.
\end{proposition}
\smallskip

$W$ will be called an involution kernel for $A$, and $A$ is called the dual potential of $A^*$. The integral kernel $K(x,y)=e^{W(x,y)}$ relates eigeinprobabilities and eigenfunctions.

We will present the proof of such results in Propositions \ref{kjr} and  \ref{tyu}. In this direction, it will be necessary to introduce the mixed skew product $\hat{\Phi}$ to be defined later (see \eqref{filo}).  Proposition \ref{kjr35} relates dual objects: the eigenprobabilities for the CA $\Phi$ and eigenfunctions for $\sigma$.  Proposition \ref{kjr} is similar, but different for Proposition 8 item (2) in \cite{CLT}, where it is considered $\hat{\sigma}$ instead of $\hat{\Phi}$.

\smallskip

Given $0< \theta < 1$, consider the maximum metric $\hat{d}_{\theta}$ defined by
$$\hat{d}_{\theta}((x,y), (x',y')):=\max( d_{\theta}(x,x'),  d_{\theta}(y,y')),$$
where
$$d_{\theta}(x,x')=\theta^{N}$$
for $N$ being the first coordinate where $x_i \neq x'_i$ (and zero if $x=x'$). The metric  $\hat{d}_{\theta}$ is equivalent to the one defined in \cite{PP}.

In the general setting, we get a matrix $M$ of zeroes and ones and deal with the subshifts of finite type
$$\hat{\Sigma}_{M}:=$$
$$\{(x,y) | x, y \in \{1,2, ..., k\}^{\mathbb{N}},\; M(x_i,x_{i+1})=1, M(y_{i-1},y_{i})=1, M(y_1,x_{1})=1\},$$
$$\Sigma_{M}:=\{x | x \in \{1,2, ..., k\}^{\mathbb{N}},\; M(x_i,x_{i+1})=1\}.$$

As usual, the ``future" and the ``past" projections are well defined respectively as
$\pi_1, \pi_2:\hat{\Sigma}_{M} \to \Sigma_{M}$ by
$$\pi_1(x,y)=x \text{ and } \pi_2(x,y)=y.$$

In this setting we define an immersion function $\varphi: \Sigma_{M} \to \hat{\Sigma}_{M}$ by
$$\varphi(x)=(x,b(x_1))$$
where $M(b_1,x_{1})=1$ and  $b(x_1) \in \Sigma_{M}$ is a fixed sequence.

Notice that the above formalism is necessary when we are dealing with actual subshifts of finite type. From now on, we assume we are considering {\bf just the full shift (for simplification of the argument)}. In this case $\Sigma_{M}=\{1,2, ..., k\}^{\mathbb{N}}$, $\hat{\Sigma}_{M}=\{1,2, ..., k\}^{\mathbb{N}} \times \{1,2, ..., k\}^{\mathbb{N}}$ and the immersion function assume a very simple form
$$\varphi(x)=(x,y')$$
where $y' \in \{1,2, ..., k\}^{\mathbb{N}}$ is a fixed sequence, for all $x$.

In order to clarify the notation we denote $\Omega:=\Sigma_{M}=\{1,2, ..., k\}^{\mathbb{N}}$, where it is necessary.

Assume $f:\Omega^2 \to \mathbb{R}$ is a Lipschitz potential. We will  consider two skew homeomorphisms  (the mixed one in Subsection \ref{uuy} and the non-mixed one in Subsection \ref{iso}), and  our main goal is to relate $f$ (via a cohomological equation) to a new potential $g:\Omega^2 \to \mathbb{R}$ such that $g(x,y)=g(x,z)$, for all $x,y,z \in \Omega$, that is, $g$ depends only of the ``future" coordinates.

\subsection{The mixed skew product $\hat{\Phi}$} \label{uuy}
Consider $\hat{\Phi}: \Omega^2 \to \Omega^2$ defined by
\begin{equation} \label{filo}\hat{\Phi}(x,y)=(\sigma(x),\tau_{x_1}(y)),
\end{equation}
where the inverse branch $\tau_j:  \Omega \to \Omega$ is given by the formula
   $$\tau_j(x)=y \text{ if } \Phi(y)=x \text{ and  } y_1=j.$$

We introduce the $n$-branch
$$\tau_{n,x}(y)=\pi_{2}(\hat{\Phi}^{n}(x,y)).$$

For a fixed $z \in \Omega$ we define the immersion function $\varphi(x)=(x,z)$ and for any $x, y \in \Omega$ we define the \textbf{mix} $W$-kernel as the \emph{formal} correspondence $W:\Omega^2 \to \mathbb{R} $ by
\begin{equation} \label{fifo}W(x,y):=\sum_{n\geq 0} f(\hat{\Phi}^{n}(x,y)) - f(\hat{\Phi}^{n}(\varphi(x))) = \sum_{n\geq 0} f(\hat{\Phi}^{n}(x,y)) - f(\hat{\Phi}^{n}(x,z)).
\end{equation}

The next result generalizes the result from \cite{PP} if we choose the trivial cellular automata (that is $\phi(x,y)=y$).

\begin{theorem} \label{ppli}
  In the above conditions, we have the following properties.
  \begin{enumerate}
    \item The skew map $\hat{\Phi}$ is a homeomorphism.
    \item The series defining $W(x,y)$ is absolutely convergent.
    \item Given  a Lipschitz function $f: \Omega^2 \to \mathbb{R}$, there exists a  Lipschitz function $g: \Omega^2 \to \mathbb{R}$ such that
        \begin{equation} \label{fifi} f(x,y) = g(x,y) + W(x,y) - W(\hat{\Phi}(x,y)),
        \end{equation}
    and $g(x,y)=g(x,z)$ for all $x,y,z \in \Omega$.
    \item $W$ is a Lipschitz function with respect to the metric $d_{\sqrt{\theta}}$ in $\Omega^2$ and ${\rm Lip}(W) \leq 2 {\rm Lip}(f) \left(\frac{1+\theta}{1-\theta}\right)$. In particular, $g$ is also a Lipschitz function with respect to the metric $d_{\sqrt{\theta}}$.
  \end{enumerate}
\end{theorem}
\begin{proof}
(1) $\hat{\Phi}$ is obviously continuous because each coordinate is. To see that it is a bijection suppose $$\hat{\Phi}(x,y)=(\sigma(x),\tau_{x_1}(y))= (\sigma(x'),\tau_{x'_1}(y'))=  \hat{\Phi}(x',y').$$
Since $\sigma(x')= \sigma(x)$, the strings $x,x'$ coincide, except eventually by the first coordinate.  On the other hand $z'=\tau_{x'_1}(y')=\tau_{x_1}(y)=z$ meaning that $\Phi(z)= y$ and $z_1=x_1$ and $\Phi(z')= y'$ and $z'_1=x'_1$ thus $x_1=x'_1$. From  the equations $\Phi(z)= y$, $\Phi(z')= y'$ and $z=z'$, we get $y=y'$ and $x=x'$.

(2) Consider the absolute value series
$$\sum_{n\geq 0} |f(\hat{\Phi}^{n}(x,y)) - f(\hat{\Phi}^{n}(x,z))| \leq  \sum_{n\geq 0} {\rm Lip}(f) \theta^n d(y,z) <\infty.$$

(3) One just needs to transform the equation
$$W(x,y) - W(\hat{\Phi}(x,y))=$$ $$= \sum_{n\geq 0} f(\hat{\Phi}^{n}(x,y)) - f(\hat{\Phi}^{n}(x,z)) - \sum_{n\geq 0} f(\hat{\Phi}^{n+1}(x,y)) - f(\hat{\Phi}^{n}(\sigma(x),z))=$$
$$=f(x,y) - f(x,z) +   {f(\hat{\Phi}(x,y))} - f(\hat{\Phi}(x,z))+   {f(\hat{\Phi}^{2}(x,y))} - $$ $$f(\hat{\Phi}^{2}(x,z))+   {f(\hat{\Phi}^{3}(x,y))} - f(\hat{\Phi}^{3}(x,z))+ $$ $$ + \ldots  - \left[  {f(\hat{\Phi}(x,y))} - f(\hat{\Phi}(\sigma(x),z))+   {f(\hat{\Phi}^{2}(x,y))} - \right. $$ $$\left. f(\hat{\Phi}^{2}(\sigma(x),z))+   {f(\hat{\Phi}^{3}(x,y))} - f(\hat{\Phi}^{3}(\sigma(x),z))+ \ldots\right]=$$
$$=f(x,y) - \left[ f(x,z) + \sum_{n\geq 1} f(\hat{\Phi}^{n}(x,z)) - f(\hat{\Phi}^{n}(\sigma(x),z))\right].$$
Thus, taking
$$g(x,y)=f(x,z) + \sum_{n\geq 1} f(\hat{\Phi}^{n}(x,z)) - f(\hat{\Phi}^{n}(\sigma(x),z))$$ we get the  cohomological equation \eqref{fifi}, where $W$ is given by \eqref{fifo}. We notice that $g(x,y)$ depends only on the future coordinates (given by $x$).

(4) Take any $x,x',y,y' \in \Omega$. 

If, for a fixed $N>0$ we suppose $d((x,y), (x',y'))=\theta^{2N}$, then
$$W(x,y)- W(x',y')=  $$
$$\sum_{n\geq 0} f(\hat{\Phi}^{n}(x,y)) - f(\hat{\Phi}^{n}(x,z)) - \sum_{n\geq 0} f(\hat{\Phi}^{n}(x',y')) - f(\hat{\Phi}^{n}(x',z)),$$
can be evaluated as follows.

For $n \leq N$ we have
$$|f(\hat{\Phi}^{n}(x,y)) - f(\hat{\Phi}^{n}(x',y'))| \leq |f(\sigma^{n}(x), \tau_{n,x}(y)) - f(\sigma^{n}(x'), \tau_{n,x'}(y'))| \leq $$
$${\rm Lip}(f) \max( d_{\theta}(\sigma^{n}(x),\sigma^{n}(x')), d_{\theta}(\tau_{n,x}(y), \tau_{n,x'}(y'))  \leq $$ $$\leq {\rm Lip}(f) \max( \theta^{-n}d_{\theta}(x,x'), \theta^{n}d_{\theta}(y, y')) \leq {\rm Lip}(f) \theta^{2N-n}.$$
A similar reasoning shows that
$$|f(\hat{\Phi}^{n}(x,z)) - f(\hat{\Phi}^{n}(x',z))| \leq {\rm Lip}(f) \theta^{2N-n}.$$

Analogously, for every $n \geq 1$ we get
$$|f(\hat{\Phi}^{n}(x,y)) - f(\hat{\Phi}^{n}(x,z))|=|f(\sigma^{n}(x), \tau_{n,x}(y)) - f(\sigma^{n}(x), \tau_{n,x}(z))|\leq$$
$$\leq {\rm Lip}(f) \max( d_{\theta}(\sigma^{n}(x),\sigma^{n}(x)), d_{\theta}(\tau_{n,x}(y), \tau_{n,x}(z))  \leq $$ $$\leq {\rm Lip}(f) \max( 0 , \theta^{n} d_{\theta}(y, y')) \leq {\rm Lip}(f) \theta^{n}.$$
A similar reasoning shows that
$$|f(\hat{\Phi}^{n}(x',y)) - f(\hat{\Phi}^{n}(x',z))|\leq {\rm Lip}(f) \theta^{n}.$$

From these four inequalities, we can write
$$|W(x,y)- W(x',y')|\leq  $$ $$\leq |\sum_{0 \leq n \leq N} f(\hat{\Phi}^{n}(x,y)) - f(\hat{\Phi}^{n}(x,z)) - \sum_{0 \leq n \leq N} f(\hat{\Phi}^{n}(x',y')) - f(\hat{\Phi}^{n}(x',z))| + $$
$$ | \sum_{ n > N} f(\hat{\Phi}^{n}(x,y)) - f(\hat{\Phi}^{n}(x,z)) - \sum_{n>N} f(\hat{\Phi}^{n}(x',y')) - f(\hat{\Phi}^{n}(x',z))| \leq$$

$$\leq\sum_{0 \leq n \leq N} |f(\hat{\Phi}^{n}(x,y)) - f(\hat{\Phi}^{n}(x',y'))| + \sum_{0 \leq n \leq N} |f(\hat{\Phi}^{n}(x,z)) - f(\hat{\Phi}^{n}(x',z))| + $$
$$ \sum_{ n > N} |f(\hat{\Phi}^{n}(x,y)) - f(\hat{\Phi}^{n}(x,z))| + \sum_{n>N} |f(\hat{\Phi}^{n}(x',y')) - f(\hat{\Phi}^{n}(x',z))| \leq$$

$$\leq\sum_{0 \leq n \leq N} {\rm Lip}(f) \theta^{2N-n} + \sum_{0 \leq n \leq N} {\rm Lip}(f) \theta^{2N-n} + $$
$$ \sum_{ n > N} {\rm Lip}(f) \theta^{n} + \sum_{n>N} {\rm Lip}(f) \theta^{n} \leq$$

$$2 {\rm Lip}(f) \left( \theta^{2N} \sum_{0 \leq n \leq N} \theta^{-n}   + \theta^{N+1} \frac{1}{1-\theta}\right)=$$ $$=2 {\rm Lip}(f) \left( \theta^{N} \sum_{0 \leq n \leq N} \theta^{-n}   +  \frac{\theta}{1-\theta}\right) \theta^{N}= C d_{\sqrt{\theta}}((x,y), (x',y')),$$
for
$$C:= 2 {\rm Lip}(f) \left(\frac{1+\theta}{1-\theta}\right)< \infty.$$
Here we used the fact that
$$\theta^{N} \sum_{0 \leq n \leq N} \theta^{-n}= \sum_{0 \leq n \leq N} \theta^{n} \leq \frac{1}{1-\theta}.$$
\end{proof}

\begin{proposition} \label{kjr} Assume that $A$ depending just of the future coordinates $x$ and $A^*$ which depends on past coordinates $y$ are such that, there exist $W:\Omega \times \Omega$ satisfying
\begin{equation} \label{pit}(A+W)(ix,y)= (A^*+W)(x, \tau_{i}(y)).
\end{equation}
Then, if $\rho_{A^*}$ is the eigenprobability for the $\Phi$-Ruelle operator $\mathcal{L}_{A^*,\Phi}$ of the potential $A^*$, then
$$\phi(x) = \int e^{W(x,y)} d \rho_{A^*}(y)$$
is the eigenfunction for the $\sigma$-Ruelle operator $ \mathcal{L}_{A,\sigma}$ of the potential $A$.

Moreover, given $A^*$, there exist $A$ and $W$ as in \eqref{pit}.
\end{proposition}

\begin{proof} Note that
$$ \mathcal{L}_A^\sigma (\phi ) (x) =  \mathcal{L}_A^\sigma (  \int e^{W(.,y)} d \rho_{A^*}(y))(x) =$$
$$\int  \mathcal{L}_A^\sigma (  e^{W(.,y)})   (x) d \rho_{A^*}(y)=\int  \sum_i e^{A(i x)} e^{W(ix,y)}  d \rho_{A^*}(y) = $$
$$ \int  \sum_i e^{A^*(\tau_i (y))} e^{W(x,iy)}  d \rho_{A^*}(y) =\int  \mathcal{L}_{A^*}^\Phi  (e^{W(x,.)})  d \rho_{A^*}(y)= $$
$$\lambda_{A^*}^\phi  \int e^{W(x,y)} d \rho_{A^*}(y).$$

From Remark \ref{infa} we get that $\lambda_{A^*}^\phi= \lambda_A^\sigma.$

For the proof of the existence of $A$ and $W$ see Proposition \ref{tyu} and Remark \ref{cent}.
\end{proof}

\begin{remark} \label{cent}
A particular situation of the previous result occurs when $f$ is defined in the following way: consider a   Lipschitz potential  $A^*: \Omega \to \mathbb{R}$ depending only on the past, and take
$$f(x,y)= A^*(\pi_2(\hat{\Phi}(x,y)))=A^*(\pi_2(\sigma(x), \tau_{1,x}(y))=A^*(\tau_{1,x}(y)) =A^* (\tau_{x_1} (y)).$$

In this case, the cohomological equation produces
  $$A^*(\tau_{1,x}(y)) = g(x,y) + W(x,y) - W(\hat{\Phi}(x,y)),$$
where $g(x,y)$ does not depends on $y$, meaning that there exists a potential $A: \Omega \to \mathbb{R}$, given by
$$A(x)=f(x,z) + \sum_{n\geq 1} f(\hat{\Phi}^{n}(x,z)) - f(\hat{\Phi}^{n}(\sigma(x),z))=$$
$$A^*(\tau_{1,x}(z)) + \sum_{n\geq 1} A^*(\tau_{n+1,x}(z)) - A^*(\tau_{n+1,\sigma(x)}(z)),$$
also Lipschitz (w.r.t. the appropriated metric) such that
$$A(x) = A^*(\pi_2(\hat{\Phi}(x,y))) + W(x,y) - W(\hat{\Phi}(x,y)),$$
for some $W$.\\
In this situation we say that $A$ and $A^*$ are dual potentials w.r.t. the $W$-kernel $W$ chosen by fixing $z$.\\
From this, we get a simpler cohomological equation
$$A^*(\tau_{1,x}(y)) = A(x) + W(x,y) - W(\hat{\Phi}(x,y)),$$
or, equivalently
\begin{equation} \label{esw}A^*(\tau_{1,\sigma(x)}(y))+ W(\hat{\Phi}(x,y)) =  A(x) + W(x,y).
\end{equation}
\end{remark}

From the above Remark, we get:

\begin{proposition} \label{tyu} Given a Lipschitz function $y \to A^* (y)$, there exists a Lipschitz function $x \to A^* (x)$, and a bi-Lipschitz function $(x,y) \to W(x,y)$, such that for all $i=1,2...,r$ and $x,y$
\begin{equation} \label{esw2}(A+W)(ix,y)= (A^*+W)(x, \tau_{i}(y)).
\end{equation}

In a similar fashion, given $A$ one can find $A^*$ and $W$ such that  \eqref{esw2} is true.
\end{proposition}

\begin{proof}
Replacing $x$ by $ix=(i,x_1,x_2,...)$ in \eqref{esw} we get
\begin{equation} \label{esw1}(A+W)(ix,y)= (A^*+W)(\sigma(ix), \tau_{1,ix}(y))= (A^*+W)(x, \tau_{i}(y)).\end{equation}
\end{proof}
\smallskip

In the above we replace $\hat{\Phi}$ by the shift $\hat{\sigma}$, and we exchange the variables $x$ and $y$ above, then we get the similar expression described by Definition 6 obtained in \cite{CLT}.\\
Finally, we notice that the dual relation can be reversed, that is,
$$A(x) = A^*(\pi_2(\hat{\Phi}(x,y))) + W(x,y) - W(\hat{\Phi}(x,y)),$$
is equivalent to
$$A(\pi_1(\hat{\Phi}^{-1}(x,y))) =  A^*(y)+ W(\hat{\Phi}^{-1}(x,y)) - W(x,y),$$
or
$$ A^*(y)= A(\pi_1(\hat{\Phi}^{-1}(x,y))) + W(x,y)- W(\hat{\Phi}^{-1}(x,y)),$$
where $\hat{\Phi}^{-1}(x,y)= (y_1x, \Phi(y))$.

\subsection{The non-mixed skew product $\hat{\Phi}_n$} \label{iso}
In this subsection we  consider  the non-mixed skew product $\hat{\Phi}_n: \Omega^2 \to \Omega^2$ defined by
$$\hat{\Phi}_n(x,y)=(\Phi(x),\tau_{x_1}(y)),$$
where the inverse branch $\tau_j:  \Omega \to \Omega$ is given by the formula
   $$\tau_j(x)=y \text{ if } \Phi(y)=x \text{ and  } y_1=j.$$

We introduce the $k$-branch
$$\tau_{k,x}(y)=\pi_{2}(\hat{\Phi}^{k}_n (x,y)).$$

For a fixed $z \in \Omega$ we define the immersion function $\varphi(x)=(x,z)$ and for any $x, y \in \Omega$ we define the \textbf{non-mixed} involution kernel $W$ as the \emph{formal} correspondence $W:\Omega^2 \to \mathbb{R} $ given by
$$W_n(x,y):=\sum_{k\geq 0} f(\hat{\Phi}^{k}_n(x,y)) - f(\hat{\Phi}^{k}_n(\varphi(x))) = \sum_{k\geq 0} f(\hat{\Phi}^{k}_n (x,y)) - f(\hat{\Phi}^{k}_n(x,z)).$$

In a similar way as in the last subsection, we  can show that such $W_n$ is well-defined. Moreover,  given a Lipschitz function  $f: \Omega^2 \to \mathbb{R}$,
there exists a Lipschitz function  $g: \Omega^2 \to \mathbb{R}$, such that,
        \begin{equation} \label{fifi44} f(x,y) = g(x,y) + W_n(x,y) - W_n(\hat{\Phi}_n(x,y)),
        \end{equation}
    and $g(x,y)=g(x,z)$ for all $x,y,z \in \Omega$.

A similar version of Theorem \ref{ppli} for $\hat{\Phi}_n$ is true (the reasoning is similar to the one in the last subsection).

\smallskip

Denote by $\mathcal{M}(\hat{\Phi}_n)$ the set of Borel probabilities on $\Omega \times \Omega$ which are invariant by $\hat{\Phi}_n.$
Given a Lipschitz potential $\hat{A} : \Omega \times \Omega \to \mathbb{R}$,  consider the topological pressure problem
\begin{equation} \label{trypol} P(\hat{A}):=\sup_{\hat{\mu} \in\mathcal{M} (\hat{\Phi}_n)} \{
h(\hat{\mu})+   \,\int \hat{A}  d \hat{\mu}\},
\end{equation}
where $h(\hat{\mu})$ is the Shannon-Kolmogorov entropy of $\hat{\mu}.$ A probability $\hat{\mu}_{\hat{A}}$  attaining the supremum value $P(\hat{A})$  will be called an $\hat{\Phi}_n$-equilibrium probability for $\hat{A}$.

Of course, if $\hat{B}: \Omega \times \Omega \to \mathbb{R}$ is such that there exist a continuous  $\hat{C}: \Omega \times \Omega \to \mathbb{R}$ satisfying
$$ \hat{B}= \hat{A} + \hat{C} - \hat{C} \circ \hat{\Phi}_n,$$
then an equilibrium probability for $\hat{A}$ is an  equilibrium probability for $\hat{B}$, and vice-versa.

The reasoning showing the validity of expression \eqref{fifi44} implies that we can replace in the above Pressure problem \eqref{trypol}  the potential $\hat{A}: \Omega \times \Omega \to \mathbb{R}$ by a Lipschitz potential $\hat{B}:\Omega \to \mathbb{R}$, that depends just on future coordinates $x$.
The equilibrium probabilities for $\hat{A}$ and $\hat{B}:=B$ will be the same. As $B$ depends just on future coordinates one can define  the Ruelle operator $\mathcal{L}_{B,\Phi}$ and take advantage of the Ruelle Theorem. Suppose $\mu_{B,\Phi}$ is the $\Phi$-equilibrium probability for $B$. This will define the probability $\mu_{B,\Phi} (\overline{a_1,a_2,...,a_m})$, $a_j\in \{1,2,...,r\}$, $j=1,2,...,m$, for any cylinder $\overline{a_1,a_2,...,a_m}$. This means that we are in fact defining probabilities for sets of the form
$$\Omega\times  \overline{a_1,a_2,...,a_m}\subset \Omega \times \Omega.$$

In \eqref{trypol} we  are interested only in probabilities on $\mathcal{M} (\hat{\Phi}_n)$. There is only way to extend $\mu_{B,\Phi}$ for a $\hat{\Phi}_n$-invariant probability $\hat{\mu}$ on $\Omega \times \Omega$. We set
$$\mu(\hat{\Phi}_n^{-k} (\Omega\times  \overline{a_1,a_2,...,a_m}  ))= \mu_{B,\Phi}( \overline{a_1,a_2,...,a_m}).$$

The probability $\mu$ is called the natural extension of   $\mu_{B,\Phi}$. Note that for our $\Phi$ (obtained from the local rule $\phi$ defined on section \ref{sec:2}) we get that $\hat{\Phi}_n^{-k} (\Omega\times  \overline{a_1,a_2,...,a_m}  )$ will exhaust the class of all possible cylinders in $\Omega\times \Omega$, changing $k$, the cylinders $\overline{a_1,a_2,...,a_m} $, etc. These cylinders will generate the Borel sigma-algebra of $\Omega \times \Omega$.
In this way, we can  identify the $\hat{\Phi}_n$-equilibrium probability for $\hat{A}$ via such $\mu$. The several ergodic properties, for instance, exponential decay of correlations, for the $\Phi$-equilibrium probability for $B$ are transferred for the $\hat{\Phi}_n$-equilibrium probability for $\hat{A}.$ In other words, we can take advantage of  the properties described in Section \ref{ther} for $\Phi$, but now  for the $\hat{\Phi}_n$-equilibrium probability for $\hat{A}.$

\section{Livsic's theorem for cellular automata} \label{Livvi}

In this section, we will present a version of Livsic's Theorem for the permutative CA $\Phi$ satisfying \eqref{perm} of Section \ref{sec:2}.    This result is based on estimates of the Birkhoff averages of the potential $A$ over periodic orbits. Among other things, we are interested in criteria to be able to decide whether a certain potential $A$ of Holder class has, or not,  as equilibrium probability
the measure of maximum entropy. We also present a criteria for finding periodic orbits for $\Phi$.

We recall that, for given $0< \theta < 1$, the usual metric  is defined by
$$d_{\theta}(x,x')=\theta^{N},$$
for $N$ being the first coordinate where $x_i=x'_i$ (and zero otherwise), makes $\Phi$ be a uniform expanding homeomorphism (by $\theta^{-1}$) and the inverse branches, $\tau_j:  \Omega \to \Omega$ are given by the formula
   $$\tau_j(x)=y \text{ if } \Phi(y)=x \text{ and  } y_1=j,$$
and are uniform contractions by $\theta$.

\begin{lemma} \label{thm:Phi expanding lemma}
  The endomorphism $\Phi: \Omega \to \Omega$ is expanding, that is, there exist $\epsilon>0$ and $L:=\theta^{-1}> 1$ such that $d(\Phi(x), \Phi(y)) \geq L d(x, y)$ for all $x, y \in \Omega$ such that $d(x, y)<\epsilon$.
\end{lemma}
\begin{proof}
  Indeed, if $d(x, y)= \theta^N<\varepsilon, x \neq y$ then $x=(a_1,...., a_{N-1}, x_{N},....)$ and $y=(a_1,...., a_{N-1}, y_{N},....)$ with $x_{N} \neq y_{N}$. As $$\Phi(x)= (\phi(a_1, a_2), \phi(a_2, a_3), ..., \phi(a_{N-2}, a_{N-1}), \phi(a_{N-1}, x_{N}),...)$$ and $$\Phi(y)= (\phi(a_1, a_2), \phi(a_2, a_3), ..., \phi(a_{N-2}, a_{N-1}), \phi(a_{N-1}, y_{N}),...)$$ we have $\phi(a_{N-1}, x_{N}) \neq \phi(a_{N-1}, y_{N})$ meaning that
  $d(\Phi(x), \Phi(y))= \theta^{N-1}= \theta^{-1} d(x,y)$. Thus we can take $L:=\theta^{-1}> 1$.
\end{proof}

\begin{theorem} \label{thm:Phi closing lemma}
   The map $\Phi$ satisfies the closing lemma property, that is, for every $\varepsilon>0$ there exists $\delta>0$ such that if $x \in \Omega$ and $n \geq 0$ are such that $d\left(\Phi^n(x), x\right)<\delta$, then there exists $y \in \Omega$ such that $\Phi^n(y)=y$ and $d\left(\Phi^k(y), \Phi^k(x)\right)<\varepsilon$ for all $0 \leq k \leq n-1$.
\end{theorem}
\begin{proof}
  Consider the orbit $\Phi^k(x)$ for $0 \leq k \leq n$  there is always $j_k$ such that $\tau_{j_k}(\Phi^k(x))=\Phi^{k-1}(x)$ thus
  $$\tau_{j_1}(\cdots \tau_{j_n}(\Phi^n(x))= x.$$
  Since $\tau_{j_1}\circ\ldots \circ\tau_{j_n}$ is a contraction by $\theta^{n}<<1$ we get a unique fixed point $y \in \Omega$ such that $\tau_{j_1}(\cdots \tau_{j_n}(y))= y.$ It is easy to see that $\Phi^n(y)=y$.

  In this way,
  $$d(y,x) = d(\tau_{j_1}(\cdots \tau_{j_n}(y)),\tau_{j_1}(\cdots \tau_{j_n}(\Phi^n(x)))) =\theta^{n} d(y,\Phi^n(x)) \leq$$ $$\leq \theta^{n} (d(y,x) + d(x,\Phi^n(x))) < \theta^{n} (d(y,x) + \delta)$$
  or
  $$d(y,x)< \frac{\theta^{n} \delta}{1- \theta^{n}}.$$
  So, if $\frac{\theta^{n} \delta}{1- \theta^{n}} < \varepsilon$ or
  $$\delta < \frac{1- \theta^{n}}{\theta^{n} } \varepsilon= (\theta^{-n} -1)\varepsilon$$
  then the proof is concluded since $\theta^{-n} -1 > 0$ is bounded away from zero for all $n$.
\end{proof}

Recall that, for a potential $A: \Omega \to \mathbb{R}$, the sum
$$S_n A(x):=\sum_{k=0}^{n-1} A(\Phi^{k}(x))$$
is well-defined.

The Walters property (w.r.t. $\Phi$) is: for every $\zeta>0$ there exist $\varepsilon>0$ such that if $x, y \in \Omega$ and $n \geq 0$ are such that $d\left(\Phi^{k}(x), \Phi^{k}(y)\right)<\varepsilon$ for all $0 \leq k \leq n-1$, then $\left|S_n A(x)-S_n A(y)\right|<\zeta$ (see \cite{CDLS}).

\begin{theorem}\label{thm:Lips implica Walters}
   If $A: \Omega \to \mathbb{R}$ is a Lipschitz (H\"{o}lder) potential then it has the Walters property w.r.t. $\Phi$.
\end{theorem}
\begin{proof}
   Consider $\zeta>0$, and $x, y \in \Omega$ and $n \geq 0$ are such that 
   $$d\left(\Phi^{k}(x), \Phi^{k}(y)\right)<\varepsilon,$$ for all $0 \leq k \leq n-1$. As $\Phi$ is expanding by $\theta^{-1}$ we know that 
   $$d\left(\Phi^{n-1}(x), \Phi^{n-1}(y)\right)<\varepsilon,$$ 
   only if $d(x,y)<\theta^{n-1} \, \varepsilon$ thus
   $$\left|S_n A(x)-S_n A(y)\right|\leq \sum_{k=0}^{n-1} \left|A(\Phi^{k}(x))-A(\Phi^{k}(y)))\right|\leq {\rm Lip}(A)  \sum_{k=0}^{n-1} \theta^{-k}d(x,y)\leq  $$
   $$< {\rm Lip}(A)  \sum_{k=0}^{n-1} \theta^{-k} \theta^{n-1} \, \varepsilon \leq \frac{{\rm Lip}(A)}{1-\theta} \varepsilon <\zeta $$
   if
   $$\varepsilon < \frac{1-\theta}{{\rm Lip}(A)}\, \zeta,$$
   concluding our proof (the H\"{o}lder case is similar).
\end{proof}

Let $X$ be a compact metric space, $T: X \rightarrow X$ a continuous map, and $A: X \rightarrow \mathbb{R}$. We recall that $A: X \rightarrow \mathbb{R}$ is a coboundary if $A=h-h\circ T$  for some continuous map $h: X \rightarrow \mathbb{R}$ and $A$ is cohomologous to $B: X \rightarrow \mathbb{R}$ if $A-B$ is a coboundary, that is, $A=B+ h-h\circ T$. In particular, any coboundary is cohomologous to zero.

An important result is
\begin{theorem}\label{thm:Livisic}
(Livsic Theorem - see Section 19.2 in \cite{Katok}) Let $X$ be a compact metric space, $T: X \rightarrow X$ a continuous map satisfying the Closing Lemma and possessing a point whose orbit is dense, and $A: X \rightarrow \mathbb{R}$ a continuous function satisfying the Walters Property. Then $A$ is a coboundary($A=h-h\circ T$) if and only if for every periodic point $x=T^n(x) \in X$, we have $S_n A(x)=0$.
\end{theorem}

Obviously, $X=\Omega$ is a compact metric space and $T=\Phi$ is a continuous map. Moreover, each Lipschitz potential is continuous, and  theorems \ref{thm:Phi closing lemma}  and \ref{thm:Lips implica Walters} ensure that the Closing lemma and Walters property are true for $\Phi$. Thus we have the following corollary.
\begin{corollary} \label{oit}
   If $A: \Omega \to \mathbb{R}$ is a Lipschitz (H\"{o}lder) potential and $\Phi$ is transitive,  then $A$ is a coboundary ($A=h-h\circ\Phi$) if and only if for every periodic point $x=\Phi^n(x) $, we have $S_n A(x)=0$.
\end{corollary}

The map  $\Phi$ of expression \eqref{pkit} is transitive as shown in Theorem \ref{thm:transitivity}. Therefore, Corollary \ref{oit} is true for such $\Phi$.

\begin{remark} \label{sozinho3}
Note that given Lipschitz (Holder) potentials $A$,$B$, and the $\Phi$-equilibrium  probabilities $m_A$ and $m_B$, there exists normalized potentials $\bar A$ and $\bar B$, such that the  equilibrium  probabilities for $\bar A$ and $\bar B$ are respectively $m_A$ and $m_B$. It follows from the end of Section 2 in \cite{Lo1}  
that  $m_A\neq m_B$, if and only if $\bar A \neq \bar B$. It follows that for a given Holder potential $A$, the equilibrium probability $m_A$ is the maximal entropy measure, if and only if, $\bar A$ is constant.
\end{remark}

\begin{proposition} If $A$ is a Lipschitz (H\"{o}lder) potential such that  for some constant $c$, is true that for any $n$ periodic orbit  $x=\Phi^n(x) \in X$, we have that $S_n (A+c)(x)=0,$ then the equilibrium probability for $A$ is the maximal entropy measure.

\end{proposition}

\begin{proof}  If for any $n$ periodic orbit  $x=\Phi^n(x) \in X$, we have that $S_n (A+c)(x)=0,$ then from Theorem \ref{thm:Livisic} we get that there exists  Lipschitz (H\"{o}lder) function $g$ such that
$$A + c = g - g \circ \Phi.$$

This shows that $A$ is $\Phi$ coboundary to a constant, and therefore the $\Phi$-equilibrium probability for $A$ is the maximal entropy measure.

\end{proof}

\begin{proposition}If $A$ is a Lipschitz (H\"{o}lder) potential such that  for a certain $n$ periodic orbit  $x=\Phi^n(x) $, we have that $S_n (A)(x)=n c_1,$ and for another  $m$ periodic orbit  $x=\Phi^m(x) $, we have that $S_m (A)(x)=m c_2,$ where $c_1\neq c_2$, then the equilibrium probability for $A$ is not the maximal entropy measure.

\end{proposition}

\begin{proof}  A   Lipschitz (H\"{o}lder) potential $A$ such the $\Phi$-equilibrium probability $\mu_A$ is the maximal entropy measure
is of the form
$$A  = g - g \circ \Phi +c,$$
where $c$ is a constant.

Therefore, for any  $k$ periodic orbit  $x=\Phi^k(x) $, we have that $S_k (A)(x)=k \,c.$
 In this way, under the above hypotheses, if $c_1\neq c_2$,  we reach  a contradiction.

\end{proof}

\begin{theorem} \label{thm:transitivity}
   If the correspondence $i \to \phi(a, i)$ is bijective for any $a \in \{1, \ldots, r\}$ then the map $\Phi$ is transitive, that is, there exists $y \in \Omega$ such that the orbit of $y$ is dense in $\Omega$.
\end{theorem}
\begin{proof}
  Consider the set of maps $\tau_j: \Omega \to \Omega$. We already know that under the hypothesis $i \to \phi(a, i)$ is bijective for any $a \in \{1, \ldots, r\}$, is true that $\Omega=\bigcup_{j} \tau_j(\Omega)$. By extension, for any $k \geq 1$ we get
  $$\Omega=\bigcup_{j_{1}\ldots j_{k}} \tau_{j_{1}}\circ\cdots\circ\tau_{j_{k}}(\Omega)$$
  a cover of $\Omega$ of diameter smaller than $\theta^k \to 0$ when $k \to \infty$. To simplify the notation we denote $\tau_{j_{1}\ldots j_{k}}:=\tau_{j_{1}}\circ\cdots\circ\tau_{j_{k}}$ and define recursively
  \[Z_{0} \in \Omega,\;\\
  Z_{k}:=\tau_{r\ldots r}\circ\cdots\circ\tau_{1\ldots 2 \, 1}\circ\tau_{1\ldots 1\, 1}(Z_{k-1})
  \]
  where the composition is take over all $k$-uples  $j_{1}\ldots j_{k}$.

  We claim that for any $j_{1}\ldots j_{k}$ and any $x \in \tau_{j_{1}}\circ\cdots\circ\tau_{j_{k}}(\Omega)$ and $w \in \Omega$ we get
  $d(x, \tau_{j_{1}\ldots j_{k}}(w)) < \theta^k$. Which is evident because each $\tau_j$ contracts by $\theta$ and $x \in \tau_{j_{1}}\circ\cdots\circ\tau_{j_{k}}(\Omega)$ means that $x=\tau_{j_{1}}\circ\cdots\circ\tau_{j_{k}}(w')$.

  Consider the sequence $Z_{k}, \, k \geq 0$ and, by the compactness of $\Omega$ a point $y \in \Omega$ for which $Z_{k_{i}} \to y$ when $i \to \infty$.

  We claim that the orbit of $y$ by $\Phi$ is dense. To see that take any $x \in \Omega$ and $\varepsilon>0$. Choose $k_{i}$ big enough to ensure $\theta^{k_{i}}< \varepsilon/2$. Notice that, by continuity, 
  $$\Phi^{m \, k_{i}}(y)=\Phi^{m \, k_{i}}(\lim_{i \to \infty} Z_{k_{i}})=\lim_{i \to \infty} \Phi^{m \, k_{i}}( Z_{k_{i}}).$$
  Let $x \in \tau_{j^{0}_{1}}\circ\cdots\circ\tau_{j^{0}_{k_{i}}}(\Omega)$, as $\Phi\circ\tau_j=Id$ we can choose $0\leq m \leq r^{k_{i}}$ in such way that
  $$\Phi^{m \, k_{i}}( Z_{k_{i}})= \tau_{j^{0}_{1}}\circ\cdots\circ\tau_{j^{0}_{k_{i}}}(w')$$
  and, by the above result we get $d(\Phi^{m \, k_{i}}( Z_{k_{i}}), x)< \theta^{k_{i}}< \varepsilon/2$.

  If we consider additionally $k_{i}$ big enough to $d( \Phi^{m \, k_{i}}(y), \Phi^{m \, k_{i}}( Z_{k_{i}}))< \varepsilon/2$, we obtain
  $$d(\Phi^{m \, k_{i}}(y), x) \leq  d( \Phi^{m \, k_{i}}(y), \Phi^{m \, k_{i}}( Z_{k_{i}}))+ d(\Phi^{m \, k_{i}}( Z_{k_{i}}), x)< \varepsilon,$$
  proving the density.


\end{proof}

\subsection{Periodic points} \label{Fi}
In this subsection, we describe an algorithmic procedure to compute the fixed points for $\Phi^m$, $m \in \mathbb{N}$. This characterization is helpful to be able to apply Livsic's Theorem.

We just assume that $i \to \phi(a, i)$ is bijective for any $a \in \{1, \ldots, r\}$.

 The inverse branches of $\Phi$ are, $\tau_j:  \Omega \to \Omega$ are given by the formula
   $$\tau_j(x)=y \text{ if } \Phi(y)=x \text{ and  } y_1=j$$
and are uniform contractions (by $\theta$).

\begin{theorem} \label{pepe}
   Suppose that $i \to \phi(a, i)$ is bijective for any $a \in \{1, \ldots, r\}$. Then,
   \begin{enumerate}
     \item Each periodic point $x$ for $\Phi$ with period $m$ is a solution of $\tau_{j_{1}}\circ\cdots\circ\tau_{j_{m}}(x)=x$ for some choice of $j_{1}, \ldots, j_{m} \in \{1, \ldots, r\}$;
     \item The periodic point, above defined, is the solution of the following recurrence
\[
\left\{
  \begin{array}{ll}
    x_1=j, \; z_{1}^{1}=i_{1},\;\ldots,\;  z_{1}^{m-1}=i_{m-1},\\
    x_{k+1}, \; z_{k+1}^{j}, 1 \leq j \leq m-1 \text{ solving }
    \left\{
      \begin{array}{ll}
        \phi(x_{k},x_{k+1})=z_{k}^{1}\\
        \phi(z_{k}^{1},z_{k+1}^{1})=z_{k}^{2}\\
        \phi(z_{k}^{2},z_{k+1}^{2})=z_{k}^{3}\\
        \cdots\\
        \phi(z_{k}^{m-1},z_{k+1}^{m-1})=x_{k}
      \end{array}
    \right.
  \end{array}
\right.
\] 
     \item The map $(j, i_{1}, \ldots, i_{m-1}) \to x$, defined as the solution of the above recurrence, is a bijection. In particular, $\Phi$ has exactly $r^m$ periodic points of period $m$.
   \end{enumerate}
\end{theorem}
\begin{proof}
(1) It is evident that each composition $\tau_{j_{1}}\circ\cdots\circ\tau_{j_{m}}$ has a unique fixed point $x$ such that $\tau_{j_{1}}\circ\cdots\circ\tau_{j_{m}}(x)=x$ because it is a contraction. Those are in fact periodic points of $\Phi$, because $\Phi^{m}(x)=x$. Reciprocally, $\Phi^{m}(x)=x$ taking  $\Phi(\Phi^{m}(x))=x$ we can find $j_1=x_1$  such that $\Phi^{m}(x)= \tau_{j_1}(x)$. Repeating this procedure we obtain $\tau_{j_{1}}\circ\cdots\circ\tau_{j_{m}}(x)=x$.

(2) To avoid extremely complex notation we consider the case $m=3$. The general case is obtained by the same reasoning. We must have
$$\tau_n\tau_i\tau_j(x)=x \text{ iff } \Phi(x)=\tau_i\tau_j(x) \text{ and  } x_1=n.$$
Let us to introduce new variables $z:=\tau_i\tau_j(x)$ and $w:=\tau_j(x)$ then
$$\tau_i(w)=z \text{ iff } \Phi(z)=w \text{ and  } z_1=i,$$
$$\tau_j(x)=w \text{ iff } \Phi(w)=x \text{ and  } w_1=j$$
and
$$\Phi(x)=z \text{ and  } x_1=n,$$
which is equivalent to the recursive system\\
$\phi(i,z_2)=w_1=j$, $\phi(z_2,z_3)=w_2$, $\phi(z_3,z_4)=w_3$, etc.\\
$\phi(j,w_2)=x_1=n$, $\phi(w_2,w_3)=x_2$, $\phi(w_3,w_4)=x_3$, etc.\\
$\phi(n,x_2)=z_1=i$, $\phi(x_2,x_3)=z_2$, $\phi(x_3,x_4)=z_3$, etc.\\
So we must solve the implicit recurrence
\[
\left\{
  \begin{array}{ll}
    x_1=n, \; z_1=i,\; w_1=j,\\
    x_{k+1}, \; z_{k+1}, \;w_{k+1} \text{ solving }
    \left\{
      \begin{array}{ll}
        \phi(x_{k},x_{k+1})=z_{k} \\
        \phi(z_{k},z_{k+1})=w_{k}\\
        \phi(w_{k},w_{k+1})=x_{k}
      \end{array}
    \right.
  \end{array}
\right.
\]

(3)  By (2) we know that, in order to find the periodic points we must solve the previous recurrence with the following initial conditions 
$$x_1=j, \; z_{1}^{1}=i_{1},\;\ldots,  z_{1}^{m-1}=i_{m-1}.$$
As we have a finite set, is enough to show that the correspondence is injective. Suppose, by contradiction, that for a different set of initial conditions
$$x_1=j, \; w_{1}^{1}=h_{1},\;\ldots,  w_{1}^{m-1}=h_{m-1}$$
we produce the same periodic point $x$. Using the respective equations we get for $k \geq 1$
$$\phi(x_{k},x_{k+1})=z_{k}^{1} \text{ and } \phi(x_{k},x_{k+1})=w_{k}^{1}$$
thus $z_{k}^{1} = w_{k}^{1}, \; k \geq 1$.
$$\phi(z_{k}^{1},z_{k+1}^{1})=z_{k}^{2} \text{ and } \phi(w_{k}^{1},w_{k+1}^{1})=w_{k}^{2}$$
thus $z_{k}^{2} = w_{k}^{2}, \; k \geq 1$.\\
And so on. In particular 
$$z_{1}^{1}=w_{1}^{1},\;\ldots,  z_{1}^{m-1}=w_{1}^{m-1}$$
contradicting $(i_{1},\;\ldots,  i_{m-1}) \neq (h_{1},\;\ldots,  h_{m-1})$.

\end{proof}

\begin{example}
  Let $\Phi$ be the map obtained by the local interaction
\[M:=
\left(
  \begin{array}{cc}
    2 & 1 \\
    1 & 2 \\
  \end{array}
\right)
\]
Fixed points ($m=1$) are obtained by $x_1=1$ and  $x_{k+1}$ solving $\phi(x_{k},x_{k+1})=x_{k}$, that is,\\
$ \phi(1,x_{2})=1$ then $x_{2}:=2$;\\
$ \phi(2,x_{3})=2$ then $x_{3}:=2$;\\
$ \phi(2,x_{4})=2$ then $x_{4}:=2$;\\
and so on, thus the first fixed point is
$(1,2,2,2,2,2,...)$.
The second one is obtained by $x_1=2$ and  $x_{k+1}$ solving $\phi(x_{k},x_{k+1})=x_{k}$, that is,\\
$ \phi(2,x_{2})=2$ then $x_{2}:=2$;\\
$ \phi(2,x_{3})=2$ then $x_{3}:=2$;\\
$ \phi(2,x_{4})=2$ then $x_{4}:=2$;\\
and so on, thus the first fixed point is
$(2,2,2,2,2,2,...)$.

For period 2 we have four points\\
a) $x_1=1,z_1=1$ and $x_{k+1}, \; z_{k+1}$ solving $\phi(x_{k},x_{k+1})=z_{k}, \; \phi(z_{k},z_{k+1})=x_{k}$. So, \\
$\phi(1,x_{2})=1, \; \phi(1,z_{2})=1$ then $x_{2}=2$ and $z_{2}=2$;\\
$\phi(2,x_{3})=2, \; \phi(2,z_{3})=2$ then $x_{3}=2$ and $z_{3}=2$;\\
and so on. Thus $(1,2,2,2,2,....)$ is our first point (it is also a fixed point).

b) $x_1=1,z_1=2$ and $x_{k+1}, \; z_{k+1}$ solving $\phi(x_{k},x_{k+1})=z_{k}, \; \phi(z_{k},z_{k+1})=x_{k}$. So, \\
$\phi(1,x_{2})=2, \; \phi(2,z_{2})=1$ then $x_{2}=1$ and $z_{2}=1$;\\
$\phi(1,x_{3})=1, \; \phi(1,z_{3})=1$ then $x_{3}=2$ and $z_{3}=2$;\\
and so on. Thus $(1,1,2,2,2,....)$ is our second point.

The remaining points are  $(2,1,2,2,2,....)$ and $(2,2,2,2,2,....)$.

Finally, we compute a periodic point of period 3. For that we choose $x_1=1,z_1=2,w_1=1$ and use the equations, \\
$\phi(x_{k},x_{k+1})=z_{k} \;  \phi(z_{k},z_{k+1})=w_{k}\; \phi(w_{k},w_{k+1})=x_{k}$.
So\\
$\phi(1,x_{2})=2 \; \phi(2,z_{2})=1\; \phi(1,w_{2})=1$ then
$x_2=1$, $z_{2}=1$ and $w_{2}=2$;\\
$\phi(1,x_{3})=1 \; \phi(1,z_{3})=2\; \phi(2,w_{3})=1$ then
$x_3=2$, $z_{3}=1$ and $w_{3}=1$;\\
$\phi(2,x_{4})=1 \; \phi(1,z_{4})=1\; \phi(1,w_{4})=2$ then
$x_4=1$, $z_{4}=2$ and $w_{4}=1$;\\
$\phi(1 ,x_{5})= 2 \; \phi( 2,z_{5})=1 \; \phi( 1,w_{5})= 1$ then
$x_5=1$, $z_{5}=1$ and $w_{5}=2$;\\
and so on, thus $(1,1,2,1,1,2,1,1,2,1,1,2,1,1,...)$ is a periodic point of period 3.
\end{example}



\section{Measure rigidity review} \label{ss6}

The {\em measure rigidity} problem for a pair of maps, that is, the problem of showing the uniqueness of the measure (generally the Parry measure) that is invariant for some pair of maps, is related to an early result by Furstenberg who proved that for relative primes $p,q\in\mathbb{Z}$ the unique infinite set of the torus $\mathbb{T}:=\mathbb{R}/\mathbb{Z}$ which is invariant for the maps $T_p:\mathbb{T}\to\mathbb{T}$ and $T_q:\mathbb{T}\to\mathbb{T}$, given by $T_p(x)= px$ and $T_q(x)=qx$, is the whole torus itself \cite{Furstenberg67}. Such result leads Furstenberg to conjecture that the Lebesgue probability measure is the unique continuous probability measure on the torus which is simultaneously invariant for the maps $T_p$ and $T_q$ (see \cite{Lyons88}).\\

Since each point of a shift space is a sequence that carries explicitly all the information about its topological location in the space, it follows that there is a direct relationship between the emergence of patterns in a given point of the shift space under the action of some map and the trajectory of this point for that map.  From the statistical point of view, by supposing $\Phi:\Omega\to \Omega$ is a cellular automaton and $\mu$ is a $\Phi$-invariant probability measure on the Borelians of $\Omega$, it means that the probability of a $\mu$-randomly chosen point $x\in \Omega$ being such that $\Phi^n(x)$ has some pattern $w_1w_2\ldots w_k$ on first $k$ positions is equal to $\mu( \overline{w_1w_2\ldots w_k})$, for all $n\geq0$.

If, additionally, $\mu$ is assumed to be also $\sigma$-invariant, then the probability that $\Phi^n(x)$ exhibits the pattern $w_1w_2\ldots w_k$ at any fixed $k$ consecutive positions, for a point $x \in \Omega$ randomly chosen according to $\mu$, is equal to $\mu( \overline{w_1w_2\ldots w_k})$. In other words,  considering the double-indexed sequence $\Big(\Phi^n(x)_i\Big)_{i,n\geq 0}$, measures that are simultaneously invariant under $\sigma$ and a cellular automaton $\Phi$ enable the identification of patterns that are spatially (in the index $i$) and temporally (in the index $n$) statistically invariant through the orbit of $x$ for $\Phi$.\\

An early result characterizing measures that are $(\sigma,\Phi)$-invariant is due to Lind \cite{Lind84}
who examined the scenario where $\mathcal{A}$ is the group $\mathbb{Z}_2$, $\Omega=\mathcal{A}^\mathbb{Z}$, $\Phi:\Omega\to\Omega$ has local rule $\phi(a,b)=a+b \pmod{2}$, and $\mu$ is any Bernoulli probability measure with full support on $\Omega$. Under these conditions, Lind demonstrated that the Cesàro mean distribution $\mathcal{C}^\Phi_N(\mu)=N^{-1}\sum_{n=0}^{N-1}\mu\circ\Phi^{-n}$ converges to the Haar measure, which here is the uniform Bernoulli probability measure $\lambda$. As consequence, since any initial Bernoulli measure is $\sigma$-invariant, and $\lambda$ is $(\sigma,\Phi)$-invariant, it follows that the uniform Bernoulli probability measure is the unique Bernoulli probability measure $(\sigma,\Phi)$-invariant\footnote{Note that given any $\sigma$-invairant meaures $\mu$, if the Cesàro mean distributions $N^{-1}\sum_{n=0}^{N-1}\mu\circ\Phi^{-n}$ converges to some probability measure $\hat\mu$ in the weak* topology, then $\hat\mu$ is $(\sigma,\Phi)$-invariant. Furthermore, it is interesting to notice that the uniform Bernoulli measures is invariant for a cellular automaton $\Phi:\mathcal{A}^{\mathbb{Z}^d}\to\mathcal{A}^{\mathbb{Z}^d}$ if, and only if, $\Phi$ is onto.}.

Latter, Schmidt \cite{Schmidt1995} considered $\mathcal{A}$ being any Abelian group, and extend the group operation from $\mathcal{A}$ to $\Omega=\mathcal{A}^{\mathbb{Z}^d}$ as a component-wise operation. Among other important results provided by the author, one can use  \cite[Corollary 29.5, p. 289]{Schmidt1995} to find out that if $\mathcal{A}=\mathbb{Z}_2$ and $\phi:\mathcal{A}^H\to\mathcal{A}$, the local rule of $\Phi$, is defined with $H\subset\mathbb{Z}$ such that $|H|\geq 2$, then the unique $(\sigma,\Phi)$-invariant probability measure with full support on $\Omega=\mathcal{A}^\mathbb{Z}$ which holds a certain mixing property (called $H$-mixing) is the uniform Bernoulli probability measure. This result can be extend for measures with full support on a subshift $G\subset\mathcal{A}^{\mathbb{Z}^d}$, where it is assumed that $G$ is a subgroup of $\mathcal{A}^{\mathbb{Z}^d}$ and $\Phi(G)=G$. In such a case,  the unique $(\sigma,\Phi)$-invariant probability measure with full support on $G$ is the Haar measure on $G$ - see \cite[Proposition 29]{Pivato2009}).\\

As Lind \cite{Lind84} and  Schmidt \cite{Schmidt1995}, as the subsequent works that have addressed the problem of measure rigidity in cellular automata, have always considered an algebraic structure on the alphabet which induces an algebraic structure on $ \mathcal{A}^{\mathbb{Z}^d}$,  $\Omega$ being a subshift and a subgroup of $\mathcal{A}^{\mathbb{Z}^d}$, and cellular automata $\Phi:\Omega\to\Omega$ which are endomorphisms for that algebraic structure. In particular, several works have considered the particular classes of cellular automata whose local rules have algebraic origin, that is, cellular automata $\Phi:\Omega\to\Omega$ with local rule $\phi:\mathcal{A}^H\to \mathcal{A}$ for some $H$ finite subset of $\subset \mathbb{Z}^d$, such that for all $x=(x_i)_{i\in H}\in \mathcal{A}^H$ we have $$\phi(x)=\sum_{i\in H}\eta_i(x_i)+c,$$
where: the sum is with respect to the operation considered in $\mathcal{A}$; $\eta_i:\mathcal{A}\to\mathcal{A}$ is an endomorphism for each $i\in H$; $\eta_i\circ\eta_j=\eta_j\circ\eta_i$ for all $i,j\in H$; and $c$ is a fixed symbol of $\mathcal{A}$. A cellular automaton in this form is said to be an {\em affine cellular automaton} and, in the particular case that $c=0$ (the identity element of $\mathcal{A}$) it is said to be a {\em linear cellular automaton}, while if $\eta_i:\mathcal{A}\to\mathcal{A}$ is the identity map for all $i\in H$ and $c=0$ it is said to be a {\em group cellular automaton}.

We remark that the algebraic origin of the cellular automata considered in all of these works is due to the need for a local rule for which one can explicitly compute successive iterates.\\

Following Lind’s seminal work \cite{Lind84}, several papers have found out measure rigidity results (see Table \ref{table_rigidity}). By assuming that $\mathcal{A}$ is some group (with some specific features), $\Omega$ is some specific shift space on $\mathcal{A}$,  $\mu$ is a measure on $\Omega$ (taken within some class of measures), and $\Phi:\Omega\to\Omega$ belongs to some specific class of endomorphic cellular automata, these works proved the convergence of the Cesaro mean distribution of $\mu$ under the dynamics of $\Phi$. 

On the other hand, following Schmidt’s approach and considering some conditions on the entropy and ergodicity, it was proved that the Haar measure (in this case the uniform Bernoulli measure) is the unique $(\sigma,\Phi)$-invariant measure if $\Phi:\mathcal{A}^\mathbb{Z}\to\mathcal{A}^\mathbb{Z}$ is a bipermutative endomorphic cellular automaton with local rule $\phi:\mathcal{A}^H\to \mathcal{A}$ for some $H\subset \mathbb{Z}$, and:
$\mathcal{A}=\mathbb{Z}/_{\mathbb{Z}_p}$ for some $p$ prime, $|H|=2$, and $\Phi$ is an affine cellular automaton \cite{HostMaassMartinez}; $\mathcal{A}$ is any finite Abelian group, and $|H|=2$ \cite{Pivato2005}; $\mathcal{A}$ is any Abelian group, and $|H|>2$ \cite{Sablik2007}. The Haar measure also was proved to be the unique $(\sigma,\Phi)$-invariant measure exhibiting some entropic and ergodic properties  for some classes endomorphic cellular automata $\Phi:\mathcal{A}^{\mathbb{Z}^d}\to\mathcal{A}^{\mathbb{Z}^d}$ on Abelian groups \cite{Einsiedler2005,Pivato2008}.\\

It is worth noting that Furstenberg's conjecture on measure rigidity for maps on the torus was proved by Lyons in \cite{Lyons88}, four years after the result by Lind \cite{Lind84} for the cellular automaton, by considering the additional hypothesis that the measure is $T_p$ exact. Furthermore, the author also found sufficient conditions under which a probability measure $\mu$ is such that $\nu \circ T_q^{-n}$ converges to the Lebesgue measure in the weak$^*$ topology. Still considering the case of the torus, in \cite{Rudolph90} and \cite{Carvalho98} it was proved that the Lebesgue measure is the unique $(T_p, T_q)$-invariant measure in a wider class of measures. In \cite{Carvalho98}, measure rigidity results were also proved for $(T, S)$-invariant measures where $T$ and $S$ are general differentiable expansive maps on the torus.

 \begin{table}[H]\tiny
\centering
\renewcommand{\arraystretch}{0.000000000001} \begin{tabular}{|p{2cm}|p{3cm}|p{3cm}|p{1.5cm}|p{.8cm}|}
\hline

 \begin{center}\bf Shift space\end{center} & \begin{center}\bf Local rule type\end{center} & \begin{center}\bf Measure class\end{center} & \begin{center}\bf Proof \end{center} & \begin{center}\bf Ref.\end{center} \\

\hline\hline

 &
 {\begin{center}\ \vspace{.1cm}\end{center}}&
{\begin{center}Bernoulli measures\end{center}} &
 &
{\begin{center} \cite{Lind84}\end{center}} \\

\cline{3-3}\cline{5-5}

{\begin{center}\vspace{-.5cm}$\mathcal{A}$ is the group $\mathbb{Z}_2$\\ \phantom{.}\\ $\Omega=\mathcal{A}^\mathbb{Z}$ \end{center}}  &
{\begin{center}\ \vspace{-.6cm} \\  Group C. A.\end{center}} &
{\begin{center}Markov measures\end{center}} &
&
{\begin{center}\cite{FerMaassMart,MaassMartinez}\end{center}}\\

\cline{1-3}\cline{5-5}

& {\begin{center}Diffusive linear C. A.\end{center}} &
{\begin{center}Harmonic mixing\end{center}} &
 {\begin{center}Harmonic analysis\end{center}}& 
{\begin{center} \cite{PivatoYassawi2002,PivatoYassawi2004}\end{center}} \\

\cline{2-3}\cline{5-5}

{\begin{center}$\mathcal{A}$ is any finite Abelian group\end{center}}
& {\begin{center}Dispersive linear C. A.\end{center}} &
{\begin{center}Dispersion mixing\end{center}} &
& 
{\begin{center} \cite{PivatoYassawi2004}\end{center}} \\

\cline{2-3}\cline{5-5}

{\begin{center} $\Omega=\mathcal{A}^\mathbb{M}$, $\mathbb{M}$ any monoid  \end{center}}& 
{\begin{center} Affine C. A.\\ \ \\ (right-permutative $\Psi$-associative and $N$-scaling)\end{center}}  &
&
 & 
{\begin{center}\cite{HostMaassMartinez}\end{center}} \\

\cline{1-2}\cline{4-5}

{\begin{center}$\mathcal{A}$ is the group $\mathbb{Z}_{p^s}$\\ \phantom{.}\\ $\Omega=\mathcal{A}^\mathbb{Z}$ \end{center}}  
& 
 &
{\begin{center}\ \vspace{-1cm} \\ \ Measures with complete connections and summable decay\end{center}}
&
 &
{\begin{center}\cite{FerMaassMartNey}\end{center}} \\

\cline{1-1}\cline{3-3}\cline{5-5}

{\begin{center}${\Omega}\subseteq(\underset{i=1}{\overset{s}\bigoplus}(\mathbb{Z}_{p^i})^{n_i})^{\mathbb{Z}^d}$\\ \ \end{center}}
& {\begin{center}\vspace{-.5cm} Linear C. A. \end{center}}
& {\begin{center}Markov measures\end{center}} & {\begin{center} Renewal theory\end{center}}
& {\begin{center}\cite{mmpy,mmpy2}\end{center}} \\

\cline{1-3}\cline{5-5}

  &
{\begin{center} Affine C. A.\end{center}} &
 &  
& 
{\begin{center}\cite{MaassMartinezSobottka2006}\end{center}} \\

\cline{2-2}\cline{4-5}

{\begin{center}\vspace{-.3cm}${\Omega}\subseteq(\underset{i=1}{\overset{s}\bigoplus}(\mathbb{Z}_{p^i})^{n_i})^{\mathbb{Z}}$\end{center}}& 
{\begin{center}Structurally-compatible C. A.\\ \ \\ (right-permutative $\Psi$-associative and $N$-scaling)\end{center}}& 
{\begin{center}\vspace{-.3cm} Measures with complete connections and summable decay\end{center}}& 
{\begin{center}\vspace{.2cm} Measure conjugacy\\  \end{center}}& 
{\begin{center}\cite{Sobottka2008}\end{center}} \\

\hline

\end{tabular}\normalsize

\caption{Different settings where the measure rigidity results have been obtained through the Cesaro mean convergence of an initial $\sigma$-invariant measure.  In the first column of the table we state the shift space and the group structure, in the second column the type of cellular automata, in the third column the class of initial measure, in the fourth column the method used to prove the convergence, and in the last column the paper were the result is proved.}\label{table_rigidity}

\end{table}

\section{Appendix}

\begin{proposition} \label{des}  Let $\phi:\mathcal{A}\times\mathcal{A} \to \mathcal{A}$ the  local rule we assumed in Section \ref{sec:2}. Consider the map $\Phi: \Omega \to \Omega$ associated with this rule. Then, for any $x \in \Omega$ the set of preimages by $\Phi$ is dense.
\end{proposition}
\begin{proof}
   Let $\Lambda_{k}(x):=\{ y \in \Omega \; | \; \Phi^k(y)=x \}$ be the set of preimages of $x$ of order $k$ and $\Lambda(x):=\bigcup_{k \geq 1}\Lambda_{k}(x)$ the set of all preimages of $x$.\\
   We will use an IFS approach to simplify the reasoning. For each $j \in \mathcal{A}$ we define the inverse branch $\tau_j:  \Omega \to \Omega$ by the formula
   $$\tau_j(x)=y \text{ if } \Phi(y)=x \text{ and  } y_1=j.$$
   In this way, the family $\tau_j, \; j \in \mathcal{A}$ is a iterated function system.\\
   Considering the fractal operator, $F(B)=\cup_{j \in \mathcal{A}} \tau_j(B)$, defined for the set of compact not empty parts of $\Omega$, and the fact that each $\tau_j, \; j \in \mathcal{A}$ is a Lipschitz contraction, there exists a unique compact set $K$ such that $F(K)=K$ and $F^k(B) \to K$ for any $B$.\\
   First, we observe that choosing $B=\{x\}$ we get $F^k(B)= \Lambda_{k}(x)$ thus $\Lambda(x)$ is dense in $K$.\\
   To conclude our proof we claim that $K= \Omega$. To do that, we will show that $F(\Omega)=\Omega$  and use the uniqueness of such set $K$. As $F(\Omega) \subseteq \Omega$, by definition, we just need to show the opposite inclusion. Take $y \in \Omega$ and $j=y_1$ we need to find $x \in \Omega$ such that $\tau_j(x)=y$, in other words,
   $$\phi(j,y_2)=x_1, \phi(y_2,y_3)=x_2, \phi(y_3,y_4)=x_3, \ldots$$
   which is always possible. Thus,  $y \in F(\Omega)$ and, since $y$ is arbitrary $F(\Omega) \supseteq \Omega$. So we have the equality.

   \smallskip

\end{proof}

\begin{proposition} \label{ell43}
   Consider $\mathcal{K}:=\{\overline{1}, \overline{2}, \ldots, \overline{r}\}$, a partition of $\Omega$. Then, for all $n \geq 0$ we have
   $$ \mathcal{K} \bigvee \Phi^{-1}(\mathcal{K}) \bigvee \ldots \bigvee \Phi^{-n}(\mathcal{K})=\biguplus_{a_{1} \ldots a_{n+1} \in \mathcal{A}} \overline{a_{1} \ldots a_{n+1}}.$$

 Moreover, give $n$, $x\in \Omega$, for each $y=(a_1,a_2,..,a_n,y_{n+1},y_{n+2},..) \in \Phi^{-n}(x)$, there exist a unique $z$ of the form $z=(a_1,a_2,...,a_n, z_{n+1},z_{n+2},..)$ in $\sigma^{-1} (x)$, and vice versa.
\end{proposition}
\begin{proof}
   The proof is by induction w.r.t. $n$. The bases, $n=0$ is the identity
   $$\mathcal{K}=\biguplus_{a_{1} \in \mathcal{A}} \overline{a_{1} }.$$ Now suppose, by hypothesis, that
   $$ \mathcal{K} \bigvee \Phi^{-1}(\mathcal{K}) \bigvee \ldots \bigvee \Phi^{-n}(\mathcal{K})=\biguplus_{a_{1} \ldots a_{n+1} \in \mathcal{A}} \overline{a_{1} \ldots a_{n+1}}.$$
   We consider a generic cylinder $\overline{b_{1} \ldots b_{n+1} b_{n+2}}$ and by applying $\Phi$  we obtain
   $$\Phi(\overline{b_{1} \ldots b_{n+1} b_{n+2}}) \subset \overline{\phi(b_{1},  b_{2}) \ldots \phi(b_{n+1}, b_{n+2})}.$$
   From our hypothesis of induction, $$ \overline{\phi(b_{1},  b_{2}) \ldots \phi(b_{n+1}, b_{n+2})} \in \mathcal{K} \bigvee \Phi^{-1}(\mathcal{K}) \bigvee \ldots \bigvee \Phi^{-n}(\mathcal{K}),$$ that is, $\overline{b_{1} \ldots b_{n+1} b_{n+2}} \in \Phi^{-1}(\mathcal{K}) \bigvee \ldots \bigvee \Phi^{-n-1}(\mathcal{K})$.\\
   Choosing $\overline{b_1} \in \mathcal{K}$ we get,
   $\overline{b_{1} \ldots b_{n} b_{n+1}} \in \mathcal{K} \bigvee \Phi^{-1}(\mathcal{K}) \bigvee \ldots \bigvee \Phi^{-n-1}(\mathcal{K})$. Thus $ \mathcal{K} \bigvee \Phi^{-1}(\mathcal{K}) \bigvee \ldots \bigvee \Phi^{-n-1}(\mathcal{K})  \supseteq \biguplus_{a_{1} \ldots a_{n+2} \in \mathcal{A}} \overline{a_{1} \ldots a_{n+2}}.$\\
   To show the opposite relation we take any $B \in \mathcal{K} \bigvee \Phi^{-1}(\mathcal{K}) \bigvee \ldots \bigvee \Phi^{-n-1}(\mathcal{K})$. We claim the $B$ is a cylinder of length $(n+1)+1=n+2$. By the induction hypothesis $ \mathcal{K} \bigvee \Phi^{-1}(\mathcal{K}) \bigvee \ldots \bigvee \Phi^{-n}(\mathcal{K})=\biguplus_{a_{1} \ldots a_{n+1} \in \mathcal{A}} \overline{a_{1} \ldots a_{n+1}}$ so
   $$B=\overline{a_{1} \ldots a_{n+1}} \cap \Phi^{-n-1}(\mathcal{K}).$$
   Notice that
   $$\Phi(y_1, y_2, y_3, ...)= (\phi(y_{1},  y_{2}), \phi(y_{2}, y_{3}),\ldots )$$
   $$\Phi^2 (y_1, y_2, y_3, ...)= (\phi(\phi(y_{1},  y_{2}),  \phi(y_{2},  y_{3})), \phi(\phi(y_{3},  y_{4}), \phi(y_{4},  y_{5})),\ldots )$$
   and so on. In this way  $\Phi^m (y)= (\psi(y_{1}, \ldots, y_{m+1}), \ldots)$, where $\psi(y_{1}, \ldots, y_{m+1})$ is a function of the first $m+1$ coordinates.\\
   As $\mathcal{K}:=\{\overline{1}, \overline{2}, \ldots, \overline{r}\}$ and $\Phi^{n+1} (y) \in \mathcal{K}$ we must have $\overline{j}$ such that $\Phi^{n+1} (y) \in  \overline{j}$. At the same time $y_1=a_{1} \ldots y_{n+1}=a_{n+1}$ because $y \in \overline{a_{1} \ldots a_{n+1}}$. In other words,
   $$\psi(y_{1}, \ldots,  y_{(n+1)+1})=j$$
   $$\psi(a_{1}, \ldots, a_{n+1},  y_{n+2})=j.$$
   Let $a_{n+2}$ be the unique solution of the above equation. Thus,
   $$B=\overline{a_{1} \ldots a_{n+1}} \cap \Phi^{-n-1}(\mathcal{K})= \overline{a_{1} \ldots a_{n+1} a_{n+2}},$$ which concludes our proof.
\end{proof}

\begin{example} \label{kli45}  An example where equation \eqref{klj1}   is true for $\Phi_1=\sigma$ and
$\Phi_2=\Phi$,  $r=2$, for  functions that depend on the first three coordinates is the following. We will have many more choices of possible functions $A_1$ and $A_2$ when compared with the last example.

Consider $\phi$ such that
$$M= \left(
\begin{array}{cc}
\phi(1,1)  & \phi(1,2)    \\
\phi(2,1)  & \phi(2,2)
\end{array}
\right)=       \left(
\begin{array}{cc}
2 & 1    \\
2 & 1
\end{array}
\right).$$

Take the functions $A_1$ and $A_2$ depending on the three first coordinates satisfying
$$ A_1 (x_1,x_2,x_3,...) =  Q_{x_1,x_2,x_3} $$
and
$$ A_2 (x_1,x_2,x_3,...) =  C_{x_1,x_2,x_3} .$$

Then,
$$  A_1 - A_1 \circ \sigma =  A_2 - A_2 \circ \Phi,$$
means
$$Q_{x_1,x_2,x_3}- Q_{x_2,x_3,x_4}= C_{x_1,x_2,x_3}- C_{\phi(x_1,x_2),\phi(x_2,x_3),\phi(x_3,x_4)}.$$

We get a solution for the system if we assume that
\begin{itemize}
  \item $C_{1,1,1} = C_{2,2,2}$
  \item $C_{1,1,2} = Q_{2,2,1}-Q_{2,2,2}+C_{2,2,2}$
  \item $C_{1,2,1} = Q_{2,1,2}-Q_{2,1,1}+Q_{1,2,2}-Q_{2,2,2}+C_{2,2,2}$
  \item $C_{1,2,2} = Q_{1,2,2}-Q_{2,2,2}+C_{2,2,2}$
  \item $C_{2,1,1} =
Q_{2,2,1}-Q_{1,1,2}+Q_{1,2,2}-Q_{2,2,2}+C_{2,2,2}$
  \item $C_{2,1,2} =
Q_{2,1,2}+Q_{2,2,1}-Q_{1,1,2}-Q_{2,2,2}+C_{2,2,2}$
  \item $C_{2,2,1} =
Q_{2,2,1}-Q_{2,1,1}+Q_{1,2,2}-Q_{2,2,2}+C_{2,2,2}$
  \item $C_{2,2,2} = \text{ free}$
  \item $Q_{1,1,1} =
-Q_{2,2,1}+Q_{2,1,1}+Q_{1,1,2}-Q_{1,2,2}+Q_{2,2,2}$
  \item $Q_{1,1,2} = \text{ free}$
  \item $Q_{1,2,1} = \text{ free}$
  \item $Q_{1,2,2} = \text{ free}$
  \item  $Q_{2,1,1} = \text{ free}$
  \item $Q_{2,1,2} = \text{ free}$
  \item $Q_{2,2,1} = \text{ free}$
  \end{itemize}
  \begin{equation} \label{three}  \bullet \,Q_{2,2,2} = \text{ free}. \end{equation}

Therefore, under the above hypotheses the
$\sigma$-equilibrium for $A_2$ is equal to the $\Phi$-equilibrium for $A_1$.

We get above a system of  $16$ linear equations, therefore,   given any choice of the $8$ parameters $C_{i,j,k}$ (defining $A_2$) we can get values $Q_{i,j,k}$ (defining $A_1$) satisfying the system. The bottom line is that any $\sigma$-equilibrium probability for  a potential $A_2$, depending  on  the first three coordinates, can be realized as a $\Phi$-invariant probability.

A particular solution of the above system would be
$$C_{1,1,1} = {\frac{1}{2}}, C_{1,1,2} = -{\frac{1}{2}},
C_{1,2,1} = -{\frac{1}{2}}, C_{1,2,2} = -{\frac{1}{2}}, C_{2,1,1} = -{
\frac{1}{2}}, C_{2,1,2} = -{\frac{1}{2}},$$ $$ C_{2,2,1} = -{\frac{1}{2}},
C_{2,2,2} = {\frac{1}{2}}, Q_{1,1,1} = 1, Q_{1,1,2} = 0, Q_{1,2,1} = 0
, Q_{1,2,2} = 0,$$ $$ Q_{2,1,1} = 0, Q_{2,1,2} = 0, Q_{2,2,1} = 0,
Q_{2,2,2} = 1.
$$

The above potential is not a potential that depends on the first two coordinates.

\end{example}

\end{document}